\documentclass[12pt,reqno]{amsart}
\usepackage{yhmath,amsmath,amsfonts,amssymb,enumerate,amsthm,appendix,caption,lipsum,xparse}
\usepackage{enumitem}
\usepackage{mathrsfs}
\usepackage{cancel}
\usepackage[utf8]{inputenc}
\usepackage[english]{babel}
\usepackage{cmap,mathtools}
\usepackage{soul}
\usepackage{graphics} 
\usepackage{epsfig} 
\usepackage{graphicx}
\usepackage{epstopdf}
\usepackage[pagewise]{lineno}

\usepackage{cmap,mathtools}
\usepackage[margin=0.85in]{geometry} 
\usepackage{cite}
\usepackage[utf8]{inputenc}
\usepackage[T1]{fontenc}
\usepackage{soul}

\usepackage{xcolor}
\usepackage{hyperref}
\usepackage[hyperpageref]{backref}
\hypersetup{
    colorlinks=true,
    linkcolor=myblue,
    filecolor=magenta,      
    urlcolor=mygreen,
    citecolor=mygreen,
}
\definecolor{mygreen}{rgb}{0.01,0.6,0.2}
\definecolor{myblue}{rgb}{0.01, 0.18, 1.0}
\newtheorem{theorem}{Theorem}
\newtheorem{proposition}{Proposition}
\newtheorem{lemma}{Lemma}

\theoremstyle{definition}
\newtheorem{definition}[theorem]{Definition}
\newtheorem{remark}[theorem]{Remark}
\newtheorem{example}[theorem]{Example}

\numberwithin{equation}{section}
\numberwithin{theorem}{section}
\numberwithin{equation}{section}
\numberwithin{theorem}{section}
\numberwithin{lemma}{section}
\numberwithin{proposition}{section}
\title[On $(p,N)$-Laplace multivalued equations with critical exponential]{ On $(p,N)$-Laplace multivalued equations with critical exponential nonlinearity in $\mathbb{R}^N$}

\author[Ankit]{Ankit}
\address[Ankit]{Department of Mathematics, Indian Institute of Technology Jodhpur, Rajasthan 342030, India}
\email{p23ma0003@iitj.ac.in}

\author[A. Sarkar]{Abhishek Sarkar}

\address[A. Sarkar]{Department of Mathematics, Indian Institute of Technology Jodhpur, Rajasthan 342030, India}
\email{abhisheks@iitj.ac.in}
\subjclass{35A15, 35B38, 35D30, 35J60, 35J92}
\keywords{$(p,N)$-Laplace, Ekeland variational principle, Orlicz spaces, Exponential growth, Non-smooth variational methods, Moser-Trudinger inequality}
 
\begin{document}
\def\dx{\mathrm{d}x}

\begin{abstract}
In this paper, we study the existence of nonnegative solutions for a class of multivalued $(p,N)$-Laplace problems having discontinuous nonlinearity with critical exponential growth in $\mathbb{R}^N$. {To demonstrate the existence results, we utilized variational methods for non-differentiable functions.} \end{abstract}  
\maketitle
\section{Introduction and Main Results}\label{Section1} 
We aim to study the existence of solutions for the following class of multivalued $(p,N)$-Laplace problems
\begin{equation}\label{main problem}
    \left\{
     \begin{aligned}
       -\Delta_pu-\Delta_Nu~&+V(x)(|u|^{p-2}u+|u|^{N-2}u)-\epsilon g(x)\in\partial_tF(x,u)  \text{ in } \mathbb{R}^N,\\
       \int_{\mathbb{R}^N} V(x)|u|^p\dx<&+\infty~, \int_{\mathbb{R}^N} V(x)|u|^N\dx<+\infty, \text{ and } u\in W^{1,p}({\mathbb{R}^N})\cap W^{1,N}({\mathbb{R}^N}),
    \end{aligned}\tag{$\mathscr{P}$}
    \right.
\end{equation}
where {$N\geq 3$, $2\leq p<N$,} $\epsilon>0$,  $\Delta_{r}u=\text{div}(|\nabla u|^{r-2}\nabla u)$ for $r\in \{p,N\}$, and $\partial_{t}F(x,t)=[\underline{f}(x,t),\overline{f}(x,t)]$,
where
$$\underline{f}(x,t)=\underset{\delta\rightarrow0}{\lim}~\underset{|s-t|<\delta}{\mathrm{ess}\inf}f(x,s) \text{ and }\overline{f}(x,t)=\underset{\delta\rightarrow0}{\lim}~\underset{|s-t|<\delta}{\mathrm{ess}\sup}f(x,s). $$
The function $V:\mathbb{R}^N\rightarrow\mathbb{R}$ { verifies} the following conditions:
\begin{enumerate}[label={($\bf V{\arabic*}$)}]
\setcounter{enumi}{0} 
    \item \label{V1} $V$ is positive continuous function and $V(x)\geq V_{0}>0$ for all $x\in \mathbb{R}^N.$
\item \label{V2} {$\bigg({V^{1/(N-1)}}\bigg)^{-1}\in L^{1}(\mathbb{R}^N)$}.
\end{enumerate}

\noindent The function $g\in \left(W^{1,p}({\mathbb{R}^N})\cap W^{1,N}({\mathbb{R}^N})\right)^*$ is such that it satisfies 
\begin{enumerate}[label={($\bf g{\arabic*}$)}]
\setcounter{enumi}{0}
    
\item \label{g1} $g$ is a non-negative function and $0<\displaystyle \underset{\mathbb{R}^N}{\int} g \dx<+\infty $.
\end{enumerate}
The nonlinearity $f:\mathbb{R}^N \times \mathbb{R} \rightarrow \mathbb{R}$ is discontinuous and has exponential critical growth, i.e., there exists $\alpha_0$, such that 
$$ \underset{|u|\rightarrow +\infty}{\lim}|f(x,u)|\mathrm{exp}({-\alpha|u|^\frac{N}{N-1}})=\begin{cases}
        0~\text{ if}~\alpha>\alpha_0\\
        +\infty~ \text{ if}~ \alpha<\alpha_0.
    \end{cases},\text{ uniformly} ~~\text{in}~x\in \mathbb{R}^N.$$
    { Next, we set $F(x,t):= \int_0^tf(x,s) \mathrm{d}s$. In this article, we will operate under the assumption that the nonlinearity \( f \) has certain characteristics. They are as follows:}
    \begin{enumerate}[label={($\bf f{\arabic*}$)}]
\setcounter{enumi}{0}
    \item \label{f1} There exists $t_0\geq0$ such that 
    $$ f(x,t) \begin{cases}= 0~\text{for}~t<t_0~\text{and}~\forall~x\in \mathbb{R}^N\\
>0~\text{for}~t>t_0~\text{and}~\forall~x\in\mathbb{R}^N. \end{cases} $$
\item \label{f2} $\underset{t\rightarrow0}{\limsup}\left(\frac {N\max\{|\xi|:\xi\in \partial_tF(x,t)\}}{{|t|^{N-1}}}\right)<\lambda_1$, uniformly for $x\in \mathbb{R}^N$, where~~ 
~~~~~~~~~$$\lambda_1=\underset{u\in \mathbf{X}\setminus\{0\}}{\inf}\frac{\|u\|_{\mathbf{X}}^N}{~~~~~~~\|u\|_{L^N(\mathbb{R}^N)}^N},$$ where $\mathbf{X}$ is defined in \eqref{p3}.
\item \label{f3} There exists a compact set $Q\subset\mathbb{R}^N$ and $\nu>N, k_1, k_2>0$ such that
$$F(x,t)\geq k_1t^{\nu}-k_2~\text{for}~t\geq0~\text{and}~\forall~x\in Q.$$
\item \label{f4} There exists $\beta>N$ such that there holds 
$$0\leq\beta F(x,t)\leq\underline{f}(x,t)t~\text{for}~t\geq t_0~\text{and}~\forall~x\in \mathbb{R}^N.$$
   \item \label{f5}
   There exist $\gamma>N$ and $\mu>0$, such that there holds
$$ F(x,t)\geq {\mu(t^{\gamma}-t_0^{\gamma})}~\text{for}~t\geq t_0~\text{and}~\forall ~x\in \mathbb{R}^N.$$
\item \label{f6}
{Assume $f$ to be a measurable and locally bounded function such that there exists $\alpha_0>0,c_1,c_2>0$ such that 
$$ |\xi|\leq c_1|t|^{N-1}+c_2\left(\mathrm{exp}({\alpha_0|t|^\frac{N}{N-1}})-\sum_{j=0}^{N-2} \frac{\alpha_0^j|t|^\frac{Nj}{N-1}}{j!}\right), \ \forall ~ \xi \in \partial_tF(x,t) \text{ and } \forall x\in \mathbb{R}^N.$$
 }
 \end{enumerate}

{ \begin{example}[Example of an $f$ satisfying \ref{f1}--\ref{f6}] Here we provide an example which satisfies all assumption on $f$. We define:
  $$f(x,t)=H(t-1)|t|^{\frac{N}{N-1}}\left(\mathrm{exp}({|t|^\frac{N}{N-1}})-\sum_{j=0}^{N-2}\frac{|t|^\frac{Nj}{N-1}}{j!}\right),~\forall~(x,t)\in~\mathbb{R}^N\times\mathbb{R},$$
     where $H$ is the Heaviside function. The above $f$ satisfies all the hypotheses \ref{f1}-\ref{f6} with $t_0=1$, $\beta=\gamma=N+1$ and $Q \subset \mathbb{R}^N$ is compact.  \end{example}}
    
    The study of partial differential equations with discontinuous nonlinearity is rapidly gaining prominence, as many free boundary problems in mathematical physics can be effectively expressed in this framework. We refer the reader to Chang \cite{MR562547} for more details. Moreover, exploring equations like \eqref{main problem} is incredibly exciting due to their wide-ranging applications in biophysics, plasma physics, and the design of chemical reactions.
    \par A rich literature on problems with discontinuous nonlinearities is available now. We refer to \cite{chang1981variational, MR1874259, MR1971108,MR1001862,radulescu1993mountain, MR3542958} and the references therein. 
    To the best of the authors' knowledge, the first instance of this consideration was in \cite{MR2126276}, where the authors considered the stationary solution of the generalized reaction-diffusion equation involving the $(p,q)$-Laplace operator with the following equation 
    $$-\Delta_p u-\Delta_q u=\lambda c(x,u),~x\in \Omega,$$ where $\Omega$ is open bounded subset of $\mathbb{R}^N$, $\lambda \in \mathbb{R}$, $1<q<p<N$. Following this work, many authors have studied the $(p,q)$-Laplace with different assumptions on nonlinearity; notable among them are \cite{MR2381659, MR2509998}. { Yang and Perera \cite{MR3483063}, for the first time},  included a bordered line case and studied the following equations with critical Trudinger-Moser nonlinearity,
   \begin{align*} \begin{cases}
     -\Delta_p u-\Delta_N u&= \mu |u|^{p-2}u+\lambda |u|^{N-2}u\mathrm{exp}({|u|^\frac{N}{N-1}})~\text{in}~\Omega,\\
     u=&0~\text{on}~\partial\Omega,
    \end{cases}\end{align*}
    where $\Omega$ is bounded open subset of $\mathbb{R}^N$, $\mu\in \mathbb{R} $ and $\lambda>0$; the existence results were included for $N/2<p<N$ and suitable restrictions on the parameters $\mu$ and $\lambda$. Zhang and Yang \cite{MR4365176} studied the existence of a positive solution of the semipositone $(p,N)$-Laplace equation in a bounded domain. Fiscella and Pucci \cite{MR4258779} showed the existence of two non-negative solutions of the $(p,N)$-Laplace equation with critical exponential nonlinearities. However, in both articles, the authors assumed the nonlinearity to be continuous. Recently, Zouai and Benouhiba \cite{MR4587597} established the existence of non-trivial solutions to the following problem
    \begin{equation*}
    -\Delta_pu-\Delta_q u=f(u)\quad\quad \text{in}\quad \mathbb{R}^N,
    \end{equation*}
    where $1\leq p,q<N$ and the nonlinearity $f$ is discontinuous.
    \par Since for $1\leq p<N$, the Sobolev embedding $W^{1,p}(\mathbb{R}^N)\hookrightarrow L^{q}(\mathbb{R}^N)$, holds for all $ q\in[p,p^*]$, where $p*=\frac{Np}{N-p}~\text{if}~ p<N$ and $p*=+\infty$ if $p\geq N$. In such situations, to solve the problems, many authors have assumed that nonlinearity follows polynomial growth at infinity. But for $p=N$, the Sobolev embedding $W^{1,N}\hookrightarrow L^{q}(\mathbb{R}^N)$, holds for all $~q\in[N,+\infty)$ and it is well known that $W^{1,N}(\mathbb{R}^N)$ is not contained in $L^{\infty}(\mathbb{R}^N)$. In this case, every polynomial growth is allowed; do \'O and de Souza \cite{MR2838280} demonstrated that this maximal growth is of exponential type. Due to this, many authors study the $N$-Laplace equation in the sense of Trudinger-Moser inequality \cite{MR301500}. The literature presents various versions and generalizations of the Trudinger-Moser inequality; interested readers are referred to \cite{MR3145759, MR3042696}.
     \par After the literature survey, we found that till now, there is no article on the multivalued $(p,N)$-Laplace problem with discontinuous nonlinearity with exponential critical growth in $\mathbb{R}^N.$ However, Alves and Santos \cite{MR3542958} studied the problem similar to \eqref{main problem} for $p=N=2$. {We are interested in determining whether results similar to \cite{MR3542958} also hold in higher dimensions}. As for $p\neq 2$, there is no Hilbert structure on the solution space $\mathbf{X}$, so we must consider gradient convergence. {Despite this, we consider the $(p,N)$-Laplace operator, which introduces a double lack of compactness, and the operator is not homogeneous, making the calculations more complicated. Moreover, we employ an approach entirely different from that used in \cite{MR3542958} to verify the Palais-Smale condition.} 
     
    \par Since the nonlinearity is of exponential type, there is no best target $L^p$ space where the elements of Clarke's generalized gradient can belong. In 1932, Orlicz introduced a special function space, which helped to overcome the difficulty (and the space is named after him as Orlicz spaces). These spaces are a generalization of $L^p$ spaces. To study the existence of a solution of \eqref{main problem}, we will use the theory of Orlicz spaces. In this article, we consider the following $N$-function (see Definition \ref{def 2.5}) for establishing our arguments 
    {\begin{equation}\label{tm1}
    \Phi_1(t)=\mathrm{exp}({  |t|^\frac{N}{N-1}})-\sum_{j=0}^{N-2} \frac{ |t|^\frac{Nj}{N-1}}{j!}.
    \end{equation}}
    We define the weighted Sobolev space in the following manner:
    \begin{equation}\label{p1}
    W_{V}^{1,r}(\mathbb{R}^N):=\bigg\{u\in W^{1,r}(\mathbb{R}^N): \int_{\mathbb{R}^N} V(x)|u|^{r}\dx<+\infty \bigg\},
    \end{equation}
    for $r \in \{p,N\}$. These spaces make sense, thanks to assumption \ref{V1}. The norm associated with space is
    \begin{equation}\label{p2}
    \|u\|_{W_{V}^{1,r}}=\left(\int_{\mathbb{R}^N} \left(|\nabla u|^r+V(x)|u|^r \right)\dx\right)^\frac{1}{r}.
    \end{equation}
    Under the norm defined in $\eqref{p2}$, the space defined in $\eqref{p1}$ is a separable and uniformly convex Banach space \cite{MR3454625}.
    Our working space will be 
    \begin{equation}\label{p3}
    \mathbf{X}={W_{V}^{1,p}(\mathbb{R}^N})\cap{W_{V}^{1,N}(\mathbb{R}^N)}.
    \end{equation}
    The space in \eqref{p3} is equipped with norm
    \begin{equation*}
    \|u\|_{\mathbf{X}}=\|u\|_{W_{V}^{1,p}}+\|u\|_{W_{V}^{1,N}}.
    \end{equation*}
    It is easy to check that $(\mathbf{X},\|\cdot\|_{\mathbf{X}})$ is reflexive and separable Banach space. For additional information regarding the properties of \(\mathbf{X}\), we refer to \cite{MR4857527}. 
    
    The following are the main results of our paper regarding the existence and multiplicity of solutions (see Definition \ref{ws}).
    \begin{theorem}\label{main result 1}
    Assume {\rm \ref{V1}-\ref{V2}, {\rm \ref{f1}-\ref{f2}, \ref{f4},\ref{f6}} and \ref{g1}} hold. Then, there exist $\hat\epsilon>0$ {and $t_*>0$ such that}, the problem (\ref{main problem}) possesses a nonnegative solution $\ u_\epsilon\in \mathbf{X}$ with { $I_{\epsilon}(u_\epsilon)=c_{\epsilon}<-\delta^*<0$, for all $\epsilon \in (0,\hat\epsilon)$ and $t_0\in[0,t_*)$}. 
    \end{theorem}
      \begin{theorem}\label{main result 2}
      Assume {\rm \ref{V1}-\ref{V2}, \ref{f1}-\ref{f6}, and \ref{g1}} hold. Then, there exists $\bar\epsilon>0$ {and $t_1>0$ such that} problem (\ref{main problem}) have a nonnegative solution $v_\epsilon\in \mathbf{X}$ such that $I_\epsilon(v_\epsilon)=b_\epsilon>0$, for all $\epsilon\in(0,\bar\epsilon)$ and $t_0 \in [0,\bar{t}_1)$  where $\bar{t}_1=\min\{t_1,t_*\}, t_*$ is defined in Theorem \ref{main result 1}.
      \end{theorem}
       \begin{theorem}\label{main result 3}
       Under the assumptions {\rm \ref{V1}-\ref{V2}, \ref{f1}-\ref{f6}, and \ref{g1}}, the problem (\ref{main problem}) have two nonnegative solutions
       $u_\epsilon$ and $v_\epsilon$ with $I_{\epsilon}(u_\epsilon)=c_{\epsilon}<0<b_\epsilon=I_\epsilon(v_\epsilon)$, for all $\epsilon \in (0,\hat{\hat{\epsilon}})$ and $t_0\in[0,\bar{t}_1)$ where $\bar{t}_1=\min\{t_1,t_*\}$, where $t_*$ is defined in Theorem \ref{main result 1} and $\hat{\hat{\epsilon}}=\min\{\bar{\epsilon},\hat{\epsilon}\}. $
       \end{theorem}
       This paper is structured as follows. Section \ref{section2} contains some basic results related to non-smooth critical point theory, a brief introduction to Orlicz spaces, Moser Trudinger type inequality for $\mathbb{R}^N$, and some embedding results. In section \ref{section 3}, we established some properties of the functional $\Upsilon$ defined in Theorem \ref{Theorem 3.1}. In sections \ref{section 4} and \ref{section5}, we obtained the nonnegative solution of the problem (\ref{main problem}) via the Ekeland Variational Principle and the Mountain Pass Theorem for non-differentiable functions, respectively.
       
       \textbf{Notations}. Throughout the paper, we will use the following notations
       \begin{itemize}
       \item ${X}\hookrightarrow Y$ denote the continuous embedding of $X$ into $Y$.
       \item ${X}\hookrightarrow\hookrightarrow Y$ denote compact embedding of ${X}$ into $Y$.
       \item $o_n(1)$ denote $o_n(1)\rightarrow0$ as $n\rightarrow+\infty$.
       \item $C_1,C_2,C_3, \cdots$ all are positive constants and may have different values at different places.
       \item $\rightharpoonup$ denotes weak convergence, $\overset{\ast}{\rightharpoonup}$ denotes weak$^{\star}$ convergence and  $\rightarrow $ denotes strong convergence.
       \item $\alpha_N=N\omega_{N-1}^\frac{1}{N-1}$, where $\omega_{N-1}$ is volume of $N-1$ dimensional unit sphere.
       \item {$m(E)$ denotes the $N$-dimensional Lebesgue measure of $E \subset \mathbb{R}^N$.
       \item $[ u > d]:= \{x \in \mathbb{R}^N: u(x) > d\}.$}
       \item $\partial f(x)$ denote Clark generalized gradient of function $f$ at $x$.
        \item $\|u\|_{L^{p}}=\left(\int_{\mathbb{R}^N}|u|^p \dx\right)^\frac{1}{p},\ 1<p<\infty.$
        \item {$\|\cdot\|_*$ denotes the norm of $g$ in the dual space of $\mathbf{X}.$}
       \item 
       $\mathrm{L}^{\tilde{\Theta}}(\mathbb{R}^N)$ denotes the Orlicz space corresponding to complementary function $\tilde{\Theta}$ associated with the $N$-function $\Theta.$
       \end{itemize}
           \section{Preliminaries}\label{section2}
    We recall some basic definitions and results, which serve as fundamental tools to establish our arguments.
    \subsection{Non smooth critical point theory}\label{s3}
    Let $({X},\|\cdot\|)$ be a real Banach space. Then, a functional $J:{X} \rightarrow \mathbb{R}$ is said to be locally Lipschitz if for every $x\in X$ there exists an open set $U$ containing $x$ such that $J|_U$ is a Lipschitz functional. We denote this class of (real-valued) functions by $Lip_{loc}(X,\mathbb{R}).$
    \begin{definition}[Generalized Directional Derivative]
    The generalized directional derivative of $J$ at $x$ in the direction $h$ is given by
     $$J^{o}(x;h)=\limsup_{y\rightarrow x,\lambda\rightarrow 0}\frac {J (y+\lambda h)-J(y)}{\lambda}.$$ \end{definition}
     The function $h\mapsto J^{0}(x,h)$ is a subadditive, continuous, and convex. 
     \begin{definition}[Generalized Gradient]\label{def2.2} A generalized gradient in the Clark sense is given by 
      $$\partial J(x)=\{x^*\in X^*: \langle x^*,h \rangle_X \leq J^{0}(x;h),\ \forall~ h \in X\}.$$ \end{definition}
     For each $x\in X$, $\partial J(x)$ is non empty, convex and weak*-compact subset of $X^*.$ Moreover, the function $$\lambda(x)=\underset{x^*\in \partial J(x)}{\min}\|x^*\|_{X^*}$$ exists and is lower semi-continuous. 
     \begin{remark} If $J\in C^{1}(X,\mathbb{R})$, then $\partial J(x)=\{J'(x)\}.$ \end{remark}
     \begin{definition}[Critical Point] A point $x_0\in X$ is called a critical point for $J$ if $0\in \partial J(x).$
     \end{definition}For more details on non-smooth critical point theory, one may refer to \cite{chang1981variational,MR367131,MR709590,do2013introduction}.
     {\begin{theorem}[ Mountain Pass Theorem for Non-differentiable
Functions {\cite[Theorem 1]{radulescu1993mountain}}]\label{thm 2.1}
          Let $X$ be a real Banach space, and $J\in Lip_{loc}(X,\mathbb{R})$ with $J(0)=0$. Assume that,
there exist $e \in X$ and $\rho, r > 0$ such that 
\begin{itemize}
    \item[(i)] $J(x)\geq \rho$ for all $x\in S_r=\{x\in X :\|x\|=r\}$,
    \item[(ii)] $J(e)< 0$ with $\|e\|>r$.    
\end{itemize}
Also, we set $$c=\underset{\gamma\in \Gamma}{\inf}\underset{t\in [0,1]}{\max}J(\gamma(t)) \text{ and } \Gamma=\{\gamma \in C([0,1],X):\gamma(0)=0,~\gamma(1)=e\}.$$
Then, $c\geq \rho$ and there exists a sequence $\{x_n\}\subset X$ which verifies $J(x_n)\rightarrow c~\text{and} ~\lambda(x_n)\rightarrow 0,$
where $\lambda(x)= \displaystyle\min_{x^*\in\partial J(x)} \|x^* \|_{X^*}$.
     \end{theorem}}
\begin{theorem}[Ekeland Variational Principle--Weak Form {\cite[Theorem 4.1]{MR1019559}}]
 Let $(X,d)$ be a complete metric space and $J:X\rightarrow\mathbb{R}$ be a lower semi-continuous function. Assume that $J$ is bounded below and set $c_1=\underset{x\in X}{\inf}J(x)$.  Then for every $\epsilon>0$, there exist $u_{\epsilon}\in X$ such that 
    $$c_1\leq J(u_{\epsilon})\leq c_1+\epsilon,$$
    $$J(x)-J(u_{\epsilon})+\epsilon d(x,u_\epsilon)>0,$$
    for every $x\in X$, $x\neq u_{\epsilon}$.  
\end{theorem}
\begin{theorem}[Lebourg's Mean-value Theorem {\cite[Theorem 3.2]{MR1454508}}]\label{Leo}
 Let $U$ be an open subset of the Banach space $X$. Let $u,v~\in U$ such that the line segment $[u,v]=\{(1-t)u+tv:0\leq t\leq 1\}\subset U$. If $f:U\rightarrow\mathbb{R}$ is locally Lipschitz function, then there exist $w\in (u,v)$ and $\zeta \in \partial f(w)$ such that
 $$f(v)-f(u)=\langle \zeta,v-u \rangle,$$ where $(u,v)=\{(1-t)u+tv:0< t< 1\}$.
\end{theorem}
\begin{proposition}[{\cite[Proposition 2.1.5]{MR709590}}]\label{proposition}
 Let $X$ be a real Banach space and $J\in Lip_{loc}(X,\mathbb{R})$.
Let $\{x_n\}\subset X $ and 
$\{\rho_n\}\subset X^*$ with $\rho_n\in\partial J(x_n)$. If $x_n\rightarrow x$ in X and $\rho_n \overset{\ast}{\rightharpoonup}\rho_0$ in $X^*$, then $\rho_0\in\partial J(x).$
\end{proposition}
     \subsection{Some version of Trudinger-Moser type inequality and related results}
     Since we will be working with exponential critical growth, some versions of the Trudinger-Moser inequality are crucial for our arguments. The initial versions were due to Trudinger \cite{trudinger1967imbeddings} and Moser \cite{Moser}. Combining their results, we can recall that if $\Omega \subset \mathbb{R}^N$ ($N \geq 2$) is a bounded domain with smooth boundary, then for any $u\in H_{0}^{1}(\Omega),$
     $$\underset{\Omega}{\int} 
     \mathrm{exp}({\alpha|u|^2})\dx<+\infty~\text{for every}~\alpha>0.$$
     Also there exists a constant $C_1=C_1(\alpha,|\Omega|)$ such that
     $$\underset{\|u\|_{H_{0}^1(\Omega)}\leq 1}\sup\underset{\Omega}{\int}\mathrm{exp}({\alpha|u|^2})\dx\leq C_1,\forall \ \alpha\leq 4\pi. $$
     {An important contribution in this direction was made by Adimurthi and Yadav \cite{MR1046704}. They showed the validity of the Trudinger-Moser-type inequality for the space $H^1(\Omega)$, where $\Omega$ is a bounded domain in $\mathbb{R} ^2$. Another notable contribution in the same direction was due to Cao \cite{cao1992nontrivial}. In this work, the author established the Trudinger-Moser inequality in the whole space $\mathbb{R}^2$. Later, the inequality was also generalized to the whole space $\mathbb{R}^N$ by do Ó \cite{MR1704875}. In this work, the author proved that for any $\alpha>0$, $N\geq 2$ and $u\in W^{1,N}(\mathbb{R}^N)$,
     \begin{equation}\label{m1}
      \underset{\mathbb{R}^N}{\int}\Phi(\alpha|u|^\frac{N}{N-1})\dx<+\infty, 
      \end{equation} where \begin{equation}\label{tm2.2}
      \Phi(t)=\text{exp}(t)-\sum_{j=0}^{N-2}\frac{t^j}{j!}.
      \end{equation} 
     Moreover, if $\|\nabla u\|_{N}\leq1$ and $
     \|u\|_{N}\leq M<+\infty$ and $\alpha<\alpha_N$ (where $\alpha_N=N\omega_{N-1}^\frac{1}{N-1}$) there exists a constant $C>0$ such that
     \begin{equation}\label{m2}
      \underset{\|\nabla u\|_N \leq1,\|u\|_N \leq M }\sup \underset{\mathbb{R}^N}{\int} \Phi\left({\alpha |u|^\frac{N}{N-1}}\right)\dx\leq C.
      \end{equation}
      \begin{remark}
          The function $\Phi_1$ defined in \eqref{tm1} is different from $\Phi$ in \eqref{tm2.2}. Although they are related by $\Phi(|t|^{N/N-1}) = \Phi_1(t)$, for all $t \in \mathbb{R}$. 
      \end{remark}}
      
      Now, we will state two lemmas without proof, which will be useful for our arguments. 
      \begin{lemma}[{\cite[Lemma 2.4]{MR2985880}}]\label{2.1}
      Let $\alpha_0>0$ and $\{u_n\}$ be sequence in $W^{1,N}(\mathbb{R}^N)$ such that 
$$\underset{n\rightarrow+\infty}{\limsup}\|u_n\|_{W^{1,N}(\mathbb{R}^N)}<\left(\frac{\alpha_N}{\alpha_0}\right)^\frac{N-1}{N}.$$
Then, for any $\alpha>\alpha_0$ close to $\alpha_0$ there exists $t>1$ ($t$ close to 1) \text{and} $C_3>0$ such that 
$$\underset{\mathbb{R}^N}{\int} \left(\Phi(\alpha|u_n|^\frac{N}{N-1})\right)^t\dx\leq C_3 ,\forall~n\in \mathbb{N}.$$
      \end{lemma}
      \begin{lemma}[{\cite[Lemma 2.3]{MR2488689}}]\label{2.2}
      Let $\alpha,M>0$ and $\alpha M<\alpha_N$ also for $u\in W^{1,N}(\mathbb{R}^N)$ such that $\|u\|_{W^{1,N}(\mathbb{R}^N)} ^\frac{N}{N-1}\leq M$. Then for any $q>N$ we have 
$$\underset{\mathbb{R}^N}{\int}|u|^q\Phi(\alpha|u|^\frac{N}{N-1})\dx\leq C_4\|u\|_{W^{1,N}(\mathbb{R}^N)}^q,$$ 
for a suitable positive constant $C_4$.
      \end{lemma}
\subsection{Orlicz spaces}\label{o1}
This subsection contains some important properties of Orlicz spaces. We define:
{\begin{definition}[$\mathbf{N}$-function]\label{def 2.5}
A function $\Theta :\mathbb{R}\rightarrow[0,+\infty)$ is an $\textbf{N}$-function if:
       \begin{itemize}
           \item[(a)] $\Theta$ is continuous and convex. 
           \item[(b)] $\Theta=0$ if and only if $t=0$.
           \item[(c)] $\Theta$ is even.
           \item[(d)] $\underset{t\rightarrow0}{\lim}\frac{\Theta(t)}{t}=0$ and $\underset{t\rightarrow +\infty}{\lim}\frac{\Theta(t)}{t}=+\infty$.
       \end{itemize} \end{definition}}
\noindent \begin{example} For example, $\Theta(t)=\mathrm{exp}({t^2})-1$ is an $\textbf{N}$-function. \end{example}{
       \begin{definition} We say that an $\textbf{N}$-function $\Theta$ satisfies $\Delta_2$-Condition if there exists $C>0$ such that
        $\Theta(2t)\leq C\Theta(t)$ for all $t\geq0$. \end{definition}
        
       \begin{definition}[Complementary/conjugate of an $N$-function]  The complementary/conjugate function $\tilde{\Theta}$ associated with $\Theta$ is given by
         \begin{equation}\label{s4}
     \tilde{\Theta}(s)=\underset{t\geq 0}{\sup}\{st-\Theta(t)\}\ \text{for } s\geq0.
     \end{equation} \end{definition}
\noindent From \eqref{s4}, we derive the  following inequality, which is given by
    \begin{equation}\label{s5}
    st\leq\Theta(t)+\tilde{\Theta}(s),\ \forall ~s,t\geq0.
\end{equation}
\begin{definition}[Orlicz space] The Orlicz space associated with $\Theta$ as
 \begin{equation}\label{s6}
       L^{\Theta}(\mathbb{R}^N)=\left\{u\in L^{1}_{loc}(\mathbb{R}^N):\underset{\mathbb{R}^N}{\int}\Theta\left(\frac{|u|}{\lambda}\right)\dx<+\infty ~\text{for some}~\lambda>0\right\}.
      \end{equation} \end{definition}
\noindent The space $L^{\Theta}(\mathbb{R}^N)$ is { a normed linear  space} endowed with following norm
      \begin{equation*}
      \|u\|_{\Theta}=\inf\left\{\lambda >0:\underset{\mathbb{R}^N}{\int}\Theta\left(\frac{|u|}{\lambda}\right)\dx\leq1\right\} .
      \end{equation*}
      The convexity of $\Theta$ implies that 
      \begin{equation}\label{s8}
      \|u\|_{\Theta}\leq 1\iff \underset{\mathbb{R}^N}{\int} \Theta(|u|)\dx\leq1.
      \end{equation}
      Some of the key properties of Orlicz spaces are listed below:
\begin{itemize}
    \item[(A)] $\left(L^{\Theta}(\mathbb{R}^N),\|\cdot\|_{\Theta}\right)$ is a Banach space.
    \item[(B)] $ L^{\Theta}(\mathbb{R}^N)$ is separable and reflexive when $\Theta$ and $\tilde{\Theta}$ satisfy $\Delta_2$ condition.
    \item[(C)]{ $\displaystyle \left|\underset{\mathbb{R}^N}{\int} uv ~\dx\right|\leq2\|u\|_{\Theta}\|v\|_{\tilde{\Theta}},\ \forall~ u\in L^{\Theta}(\mathbb{R}^N),~v\in L^{\tilde{\Theta}}(\mathbb{R}^N).$}
    \item[(D)] From the above property, we have the following relation,
    $$ L^{\tilde{\Theta}}(\mathbb{R}^N)\subset \left(L^{\Theta}(\mathbb{R}^N)\right)^*.$$
    \item[(E)] {We define} $$E_{\Theta}(\mathbb{R}^N)=\overline{\left\{u\in L^1_{loc}(\mathbb{R}^N):u \text{ having bounded support on }\mathbb{R}^N\right\}}^{\|\cdot\|_{\Theta}},$$
    then equivalently, we also have $$E_{\Theta}(\mathbb{R}^N)=\overline{C_c^{\infty}(\mathbb{R}^N)}^{\|\cdot\|_{\Theta}}. $$
    \item[(F)] The dual of $E_{\Theta}(\mathbb{R}^N)$ is isomorphic to $L^{\tilde\Theta}(\mathbb{R}^N)$.    \end{itemize}
For proofs of the above-stated properties and more details on Orlicz spaces, we refer to \cite{MR1113700, MR2424078}.}
\subsection{Embedding Results}
In this subsection, we present some important embedding results that will be used later.
\begin{lemma}\label{Lemma 2.3}
Suppose \ref{V1} holds. Then, the following embedding is continuous
$$\mathbf{X}\hookrightarrow W^{1,p}(\mathbb{R}^N)\cap W^{1,N}(\mathbb{R}^N)\hookrightarrow L^{q}(\mathbb{R}^N),$$
for any $q\in [p,p^*]\cup[N,+\infty).$
\end{lemma}
\begin{proof} Due to assumption \ref{V1}, for any $u\in W_V^{1,p}(\mathbb{R}^N)$, we get
$$\underset{\mathbb{R}^N}{\int}(|\nabla u|^p+V(x)|u|^p)\dx\geq \underset{\mathbb{R}^N}{\int}(|\nabla u|^p+V_0|u|^p)\dx\geq \min\{1,V_0\}\underset{\mathbb{R}^N}{\int}(|\nabla u|^p+|u|^p)\dx.$$
It implies$$\|u\|_{W_{V}^{1,p}}^p\geq \min\{1,V_0\}\|u\|_{W^{1,p}}^p.$$
So, the following embedding is continuous
\begin{equation}\label{e1}
W_{V}^{1,p}(\mathbb{R}^N)\hookrightarrow W^{1,p}(\mathbb{R}^N)\hookrightarrow L^q(\mathbb{R}^N) \text{  for any } q\in[p,p^*].
\end{equation}
Similarly, we have
\begin{equation}\label{e2}
W_{V}^{1,N}(\mathbb{R}^N)\hookrightarrow W^{1,N}(\mathbb{R}^N)\hookrightarrow L^q(\mathbb{R}^N) \text{ for any } q\in[N,+\infty).
\end{equation}
Hence from $\eqref{e1}$ and $\eqref{e2}$, the proof of the lemma follows.\end{proof}
\begin{lemma}\label{2.4}
Suppose \ref{V1}-\ref{V2} hold. Then, the following embedding  is compact
$$\mathbf{X}\hookrightarrow\hookrightarrow L^{q}(\mathbb{R}^N), \forall \ q\geq 1.$$
\end{lemma}
\begin{proof} Since we have 
\begin{equation}\label{e3}
\mathbf{X}\hookrightarrow W_{V}^{1,N}(\mathbb{R}^N).
\end{equation}
From \cite[Lemma 2.4]{MR2873855}, we have
\begin{equation}\label{e4}
W_{V}^{1,N}(\mathbb{R}^N)\hookrightarrow\hookrightarrow L^{q}(\mathbb{R}^N),\forall\ q\geq1.
\end{equation}
From \eqref{e3} and \eqref{e4}, the conclusion follows immediately. \end{proof}
{In the next lemma, we establish a relation between $\Phi_1$ (defined in \eqref{tm1}) and its complementary function $\tilde{\Phi}_1.$}
\begin{lemma}\label{e5}
Let $\zeta(t)=\max\{t,t^N \}$ and
 $\tilde{\Phi}_{1}$ be conjugate function associated with $\Phi_{1}$. Then, the following inequalities are satisfied:
\begin{itemize}
\item[(i)]  $\tilde{\Phi}_{1}(\frac{\Phi_{1}(r)}{r})\leq \Phi_{1}(r),~\forall~ r>0$.
\item[(ii)] $\tilde{\Phi}_{1}(tr)\leq \zeta(t)\tilde{\Phi}_{1}(r), ~\forall~ t,r \geq 0.$ 
\end{itemize}
Hence, $\tilde{\Phi}_{1}\in \Delta_2$ and $E_{\tilde{\Phi}_{1}}(\mathbb{R}^N)=L^{\tilde{\Phi}{1}}(\mathbb{R}^N)$.
\end{lemma}
\begin{proof} { The proof of (i) follows from \cite[Lemma A.2]{MR2271234}.
Next, we proceed to prove (ii). Consider the function $h(t)=\frac{t\Phi_1'(t)}{\Phi_1(t)}$. We note that due to the convexity of $\Phi_1$, the function  $h$ is increasing for all $t\ge 1.$ Hence,
\begin{equation*}
\frac{N}{N-1}<h(1)=\frac{N}{N-1}\frac{\Phi_1'(1)}{\Phi_1(1)}<\frac{\Phi'_{1}(t)t}{\Phi_{1}(t)},~\forall t\geq1.
\end{equation*}
This leads to 
$$\frac{N}{N-1}\leq \frac{\Phi'_{1}(t)t}{\Phi_{1}(t)},\ \forall~   t\in(1,+\infty).$$}
Fix $s>0$ such that $t=\tilde{\Phi}_{1}'(s)$, from \cite[Lemma 2.5]{MR2271234}, we have  $$s\Phi^{'}_{1}(s)=\tilde{\Phi}_{1}(s)+\Phi_{1}(\tilde{\Phi}^{'}_{1}(s)).$$
Differentiating the above and using $\Phi_{1}''(s)\neq0$, we obtain $\tilde{\Phi}_{1}'=(\Phi_1')^{-1}$.
Therefore, we get
$$ \frac{N}{N-1}\leq \frac{s\Phi'_{1}(s)}{s\tilde{\Phi}'_{1}(s)-\tilde{\Phi}_{1}(s)}.$$
This further implies
\begin{equation}\label{e6}
\frac{\tilde{\Phi}_{1}^{'}(s)}{\tilde{\Phi}_{1}(s)}\leq\frac{N}{s}.
\end{equation} For $t>1$ and $r>0$,  on integrating \eqref{e6} from $r$ to $rt$, we have
\begin{equation}\label{e7}
\tilde{\Phi}_{1}(rt)\leq t^{N} \tilde{\Phi}_{1}(r),~\forall ~t>1,r\geq0.
\end{equation}
Since $\tilde{\Phi}_{1}(0)=0$ and $\tilde{\Phi}_{1}$ is convex, so
$$1\leq\frac{\tilde{\Phi}_{1}^{'}(t)t}{\tilde{\Phi}_{1}(t)},~t\in(0,\infty).$$
For $0<t\leq 1$, on integrating \eqref{e6} from $rt$ to $r$, we have
\begin{equation}\label{e8}
\tilde{\Phi}_{1}(rt)\leq t \tilde{\Phi}_{1}(r),~\forall ~t\in(0,1],r\geq0.
\end{equation}
From \eqref{e7} and \eqref{e8}, (ii) follows immediately. Hence, $\tilde{\Phi}_{1}\in\Delta_2$ and $E_{\tilde{\Phi}_{1}}(\mathbb{R}^N)=L^{\tilde{\Phi}_{1}}(\mathbb{R}^N)$. \end{proof}
\begin{lemma}\label{2.6}
Let $E_{\Phi_{1}}(\mathbb{R}^N)$ be a subspace of an Orlicz space $L^{\Phi_{1}}(\mathbb{R}^N)$ defined in subsection~\ref{o1}. Then, the following embeddings hold:
\begin{itemize}
\item[(i)] $\mathbf{X}\hookrightarrow\ E_{\Phi_{1}}(\mathbb{R}^N)$, and
\item[(ii)] $E_{\Phi_{1}}(\mathbb{R}^N)\hookrightarrow L^{N}(\mathbb{R}^N).$
\end{itemize}
\end{lemma}
\begin{proof} $(i)$ Let $u\in \mathbf{X}$, then thanks to \eqref{m1}, we have $u\in L^{\Phi_{1}}(\mathbb{R}^N)$. Hence
$$\mathbf{X}\subset L^{\Phi_{1}}(\mathbb{R}^N).$$
From \eqref{m2}, we have
\begin{equation*}
\underset{\mathbb{R}^N}{\int} \Phi_{1}\left(\frac{|u|}{\|u\|_{W^{1,N}}}\right)\dx \leq C_2.
\end{equation*}
Now, the convexity of $\Phi_{1}$ and $\Phi_{1}(0)=0$ together imply that for some $C>1$, we have 
\begin{equation*}
\underset{\mathbb{R}^N}{\int} \Phi_{1}\left(\frac{|u|}{C\|u\|_{W^{1,N}}}\right)\dx \leq 1.
\end{equation*}
On combining \eqref{s8} and \eqref{e1}, we have $\mathbf{X}\hookrightarrow L^{\Phi_{1}}(\mathbb{R}^N).$ From \cite{MR3542958} it follows that $\mathbf{X}\hookrightarrow E_{\Phi_{1}}(\mathbb{R}^N) $. Hence, the proof of (i) follows. \\
(ii) { Since the following inequality holds
$$\frac{|t|^N}{(N-1)}\leq \sum_{j=N-1}^{\infty} \frac{|t|^\frac{Nj}{N-1}}{j!},$$
then, for each $u\in E_{\Phi_{1}}(\mathbb{R}^N)$, we have
$$ \frac{1}{N-1}\underset{\mathbb{R}^N}{\int}\left(\frac{|u|}{\|u\|_{\Phi_{1}}}\right)^N \dx\leq \underset{\mathbb{R}^N}{\int}\Phi_{1}\left(\frac{|u|}{\|u\|_{\Phi_{1}}}\right) \dx\leq 1.$$}
Hence $E_{\Phi_{1}}(\mathbb{R}^N)\hookrightarrow L^{N}(\mathbb{R}^N),$ and the proof is complete. \end{proof}
\section{Properties of functional \texorpdfstring{$\Upsilon$}{Lg}}\label{section 3}
In this section we will show that the functional  $\Upsilon:\mathbf{X}\rightarrow\mathbb{R}$ given by
    $$\Upsilon(u)=\underset{\mathbb{R}^N}{\int}F(x,u)\dx$$ 
    is well-defined. We consider the functional \texorpdfstring{$\Upsilon$}{Lg} in a more appropriate domain to apply variational methods. One possible way to define it is the following 
    $$\Upsilon:L^{\Phi_1}(\mathbb{R}^N)\rightarrow\mathbb{R}.$$
    But, once $\Phi_1$ does not satisfy $\Delta_2$ condition, we cannot guarantee that $I\in (L^{\Phi_1}(\mathbb{R}^N))^*$, where we define
    \begin{equation*}
    I(u)=\underset{\mathbb{R}^N}{\int}uv\dx,~\forall~ u\in L^{\Phi_1}(\mathbb{R}^N),
    \end{equation*}
    for some measurable function $v.$ Using the fact that  dual of $E_{\Phi_1}(\mathbb{R}^N)$ is isomorphic to $L^{\tilde\Phi_1}(\mathbb{R}^N)$ and by Lemma \ref{e5}, it is appropriate to define the functional $\Upsilon$ on $E_{\Phi_1}(\mathbb{R}^N).$ Next, we show the functional $\Upsilon$ is locally Lipschitz.
\begin{theorem}\label{Theorem 3.1}
Assume  \ref{f6} hold. The functional  $\Upsilon:E_{\Phi_{1}}(\mathbb{R}^N)\rightarrow\mathbb{R}$ given by
    $$\Upsilon(u)=\underset{\mathbb{R}^N}{\int}F(x,u)\dx$$ 
    is well defined and  $\Upsilon \in Lip_{loc}(E_{\Phi_{1}}(\mathbb{R}^N),\mathbb{R)}.$
    \end{theorem}
    \begin{proof} Let $u\in E_{\Phi_{1}}(\mathbb{R}^N)$ and $R>0$. Consider $w,v\in B_{R}(u)\subset E_{\Phi_{1}}(\mathbb{R}^N) $, where $ B_{R}(u)$ denotes the open ball of radius $R$ in $E_{\Phi_{1}}(\mathbb{R}^N).$  By Theorem~(\ref{Leo}), there is $\xi\in \partial_{t}F(x,\theta)$ with $\theta\in[w,v]$ such that 
    \begin{equation*}
    |F(x,w)-F(x,v)|=|\langle\xi,w-v\rangle|\leq|\xi\|w-v|.
\end{equation*}
By assumption \ref{f6}, we have $\alpha_0, c_1, c_2>0$ such that 
   {\begin{equation*}
   |\xi|\leq c_1|\theta|^{N-1}+c_2\left(\mathrm{exp}({\alpha_0|\theta|^\frac{N}{N-1}})-\sum_{j=0}^{N-2} \frac{\alpha_0^j|\theta|^\frac{Nj}{N-1}}{j!}\right). \end{equation*}}
   For simplicity, denote  $\Phi(\alpha_0|t|^{\frac{N}{N-1}})=\mathrm{exp}({\alpha_0|t|^\frac{N}{N-1}})-\sum_{j=0}^{N-2} \frac{\alpha_0^j|t|^\frac{Nj}{N-1}}{j!}.$ 
   Then, for $\alpha>\alpha_0$ and $\alpha$ close to $\alpha_0$, we have
   {\begin{equation*}
   |F(x,w)-F(x,v)|\leq \left(c_1|\theta|^{N-1}+c_2\Phi(\alpha|\theta|^{\frac{N}{N-1}}\right)|w-v|.
    \end{equation*}}
\noindent Now we set $\eta(x)=|w(x)|+|v(x)|$. Since $\theta\in[w,v]$ and $\Phi$ is an increasing and convex function, it follows that
   \begin{equation*}
   |F(x,w)-F(x,v)|\leq \left(c_1|\eta|^{N-1}+c_2\Phi(\alpha|\eta|^{\frac{N}{N-1}}\right)|w-v|,
    \end{equation*}
    and so
    \begin{equation}\label{3}
    |\Upsilon(w)-\Upsilon(v)|\leq \underset{\mathbb{R}^N}{\int}\left(c_1|\eta|^{N-1}+c_2\Phi(\alpha|\eta|^{\frac{N}{N-1}}\right)|w-v|\dx.
     \end{equation}
From \eqref{3}, by H\"older's inequality and Lemma \ref{2.6} (ii), we get
\begin{equation*}
    |\Upsilon(w)-\Upsilon(v)|\leq c_1\|\eta\|_{L^{N}}^{N-1}\|w-v\|_{\Phi_{1}}+c_2\left(\underset{\mathbb{R}^N}{\int}\Phi\left(\frac{\alpha N}{N-1}|\eta|^{\frac{N}{N-1}}\right)dx\right)^{\frac{N-1}{N}}\|w-v\|_{\Phi_{1}}.
      \end{equation*}
      It is easy to see that 
      \begin{equation}\label{a1}
      \|\eta\|_{\Phi_{1}}^{N-1}\leq 2^{2N-4}\|w-u\|_{\Phi_{1}}^{N-1}+2^{2N-4}\|v-u\|_{\Phi_{1}}^{N-1}+2^{3N-5}\|u\|_{\Phi_{1}}^{N-1}.
      \end{equation}
      Since $\Phi$ is an increasing and convex function, we have
      \begin{align*}
      \underset{\mathbb{R}^N}{\int}\Phi\left(\frac{\alpha N}{N-1}|\eta|^{\frac{N}{N-1}}\right)\dx \notag &\leq \frac{1}{4}\underset{\mathbb{R}^N}{\int}\Phi\left(\frac{4^\frac{N}{N-1}\alpha  N}{N-1}|w-u|^{\frac{N}{N-1}}\right)\dx\\ \quad&+\frac{1}{4}\underset{\mathbb{R}^N}{\int}\Phi\left(\frac{4^\frac{N}{N-1}\alpha  N}{N-1}|v-u|^{\frac{N}{N-1}}\right)\dx+\frac{1}{2}\underset{\mathbb{R}^N}{\int}\Phi\left(\frac{4^\frac{N}{N-1}\alpha  N}{N-1}|u|^{\frac{N}{N-1}}\right)\dx.
      \end{align*}
      Now fix $R>0$ such that $R<\left(\frac{1}{4^\frac{N}{N-1}\frac{\alpha N}{N-1}}\right)^\frac{N-1}{N}.$ Then, we have $$\bigg\|\left({4^\frac{N}{N-1}\frac{\alpha N}{N-1}}\right)^\frac{N-1}{N} (w-u)\bigg\|_{\Phi_{1}}\leq 1 \text{ and }\bigg\|\left({4^\frac{N}{N-1}\frac{\alpha N}{N-1}}\right)^\frac{N-1}{N} (v-u)\bigg\|_{\Phi_{1}}\leq 1, ~\forall \ w,v\in B_{R}(u).$$
      From \eqref{3}, \eqref{a1}, \eqref{s5} and \eqref{s8}, we have 
      \begin{equation*}
      |\Upsilon(w)-\Upsilon(v)|\leq C \|w-v\|_{\Phi_{1}}, ~\forall ~w,v\in B_{R}(u),
      \end{equation*}
    where $C=C(R,u)$ is a positive constant. Hence  $\Upsilon \in Lip_{loc}(E_{\Phi_{1}}(\mathbb{R}^N),\mathbb{R)}.$ \end{proof}
   {Now, we intend to show the following generalized sub-differential inclusion
   $$\partial \Upsilon(u)\subset\partial_tF(x,u), \ \forall u\in E_{\Phi_{1}}(\mathbb{R}^N).$$
   This inclusion will be crucial to implement the variational methods.}
   To do this, we need the following result.
   \begin{lemma}[{\cite[Lemma 4.1]{MR3542958}}]{\label {orlicz}}
   Let $\phi:\mathbb{R}\rightarrow\mathbb{R_{+}}$ be an $N$-function and
   \begin{equation*}
   h_n\rightarrow h ~~in ~~E_{\phi}(\mathbb{R}^N).
   \end{equation*}
   Then, there is $\hat{h}\in E_{\phi}(\mathbb{R}^N)$ and a subsequence of $\{h_n\}$, denoted by $\{h_{n_k}\}$, such that
   \begin{itemize}
   \item[(i)] $h{_{n_k}}(x) \rightarrow h(x)$, \text{a.e.} $x\in \mathbb{R}^N,$
   \item[(ii)] $|h{_{n_k}}(x)|\leq \hat{h}(x)$ \text{a.e.} $x\in \mathbb{R}^N.$
   \end{itemize}
   \end{lemma}
  { \begin{theorem}
   Assume \ref{f6} holds. Also, assume $ \underline{f}(x,t)$ and $\overline{f} (x,t)$ are $N$-measurable function. Then, for each $u\in E_{\Phi_{1}}(\mathbb{R}^N)$,
   \begin{equation*}
   \partial \Upsilon(u)\subset\partial_tF(x,u)=[\underline{f}(x,u(x)),\overline{f}(x,u(x))] \ \text{a.e. in} ~\mathbb{R}^N.
   \end{equation*}
Moreover, $\partial\Upsilon|_{\mathbf{X}}(u)\subset\partial\Upsilon(u)$ for all $u\in \mathbf{X}.$
\end{theorem}}
\begin{proof} For  $u, w\in E_{\Phi_{_1}}(\mathbb{R}^N)$, let ${h_{j}}\subset E_{\Phi_{1}}(\mathbb{R}^N)$ with $h_j\rightarrow 0$ in $E_{\Phi_{1}}(\mathbb{R}^N)$ and $\{\lambda_j\}\subset\mathbb{R_{+}}$ such that $\lambda_j\rightarrow 0$, then 
\begin{equation*}
\Upsilon^{o}(u;w)=\underset{j\rightarrow+\infty}{\lim}\underset{\mathbb{R}^N}{\int}\frac{F(u+h_{j}+\lambda_j w)-F(u+h_{j})}{\lambda_j}\dx.
\end{equation*}
For simplicity, we denote
\begin{equation*}
F_{j}(u,w)=\frac{F(u+h_{j}+\lambda_j w)-F(u+h_{j})}{\lambda_j}.
\end{equation*}
Here, we can apply Lebourg's theorem (see {\cite[Theorem 2.3.7]{MR709590}}) which guarantees that there exists $\xi_j\in \partial_t F(x,\theta_j)$, with $\theta_j \in [u+h_{j}+\lambda_j v, u+h_{j}]$ so that 
\begin{equation*}
|F_{j}(u;w)|=\frac{|\langle\xi_j,\lambda_j w\rangle|}{\lambda_j}\leq |\xi_j\|w|.
\end{equation*}
Then by \ref{f6}, for $\alpha>\alpha_0$
\begin{equation*}
|F_{j}(u;w)|\leq |\xi_j\|w|\leq \left(c_1|\theta_j|^{N-1}+c_2\Phi(\alpha|\theta_j|^\frac{N}{N-1})\right)|w|.
\end{equation*}
Now fixing {$\gamma_j=(|u|+|h_j|+\lambda_j |w|)+(|u|+|h_j|)=2|u|+2|h_j|+\lambda_j|w|$}, we deduce 
\begin{equation*}
|F_{j}(u;w)|\leq |\xi_j\|w|\leq \left(c_1|\gamma_j|^{N-1}+c_2\Phi(\alpha|\gamma_j|^\frac{N}{N-1})\right)|w|.
\end{equation*}
{ Since $h_j\rightarrow 0$ in $E_{\Phi_{1}}(\mathbb{R}^N)$, by Lemma \ref{orlicz}, up to a subsequence (still denoted by the same), there exists $\hat{h}\in E_{\Phi_{1}}(\mathbb{R}^N)$} such that $|h_j|\leq \hat{h}$ and $\lambda_j\rightarrow 0$ in $\mathbb{R}$ implies
\begin{align*}
\gamma_j(x)&=(|u(x)|+|h_j(x)|+\lambda_j |w(x)|)+(|u(x)|+|h_j(x)|)\\ & \leq2|u(x)|+2|\hat{h}(x)|+c_3|w(x)| \text{ a.e. in }\mathbb{R}^N,
\end{align*}
for some positive constant $c_3.$
By H\"older's inequality and using the embedding $E_{\Phi_{1}}(\mathbb{R}^N)\hookrightarrow L^{N}(\mathbb{R}^N)$ (Lemma \ref{2.6} (ii)), we get
\begin{equation*}
(2|u(x)|+2|\hat{h}(x)|+c_3|w(x)|)^{N-1}w(x)\in L^{1}(\mathbb{R}^N) \text{ and } \Phi(\alpha|\gamma_j|^\frac{N}{N-1})w \in L^1(\mathbb{R}^N).
\end{equation*}
By the Lebesgue dominated convergence theorem, we further get
\begin{align*}
\Upsilon^{o}(u;w)\notag &=\underset{j\rightarrow+\infty}{\lim}\underset{\mathbb{R}^N}{\int}\frac{F(u+h_{j}+\lambda_j w)-F(u+h_{j})}{\lambda_j}\dx\\ \notag &=\underset{\mathbb{R}^N}{\int}\underset{j\rightarrow+\infty}{\lim}\frac{F(u+h_{j}+\lambda_j w)-F(u+h_{j})}{\lambda_j}\dx\\ \notag 
&\leq\underset{\mathbb{R}^N}{\int}F^{0}(u,w)\dx=\underset{\mathbb{R}^N}{\int} \max\{\langle\xi,w\rangle: \xi\in\partial_tF(x,u)\}\dx\\  &\leq \underset{[w<0]}{\int}\underline{f}(x,u)w\dx+ \underset{[w>0]}{\int} \overline{f}(x,u)w\dx.
\end{align*}
Since $\left(E_{\Phi_{1}}(\mathbb{R}^N)\right)^*=L^{\tilde{\Phi}_{1}}(\mathbb{R}^N)$, so for each $\xi\in\partial\Upsilon(u)\subset (E_{\Phi_{1}}(\mathbb{R}^N))^*$, the function $\tilde{\xi}\in L^{\tilde{\Phi}_{1}}(\mathbb{R}^N) $ such that
\begin{equation*}
\langle\xi,v \rangle=\underset{\mathbb{R}^N}{\int} \tilde{\xi}v\dx, \ \forall v\in E_{\Phi_{1}}(\mathbb{R}^N).
\end{equation*}
Now our aim  is to show for any $\tilde{\xi}\in E_{\Phi_{1}}(\mathbb{R}^N)$, we have, 
\begin{equation*}
\tilde{\xi}(x)\in[\underline{f}(x,u(x)),\overline{f}(x,u(x))] \text{ a.e. in } \mathbb{R}^N.
\end{equation*}
Suppose there exist a set $\textbf{B}$ such that $0<m(\textbf{B})<+\infty$ such that 
\begin{equation}\label{z1}
\tilde{\xi}(x)>\overline{f}(x,u(x)), \ x \in\textbf{B}.
\end{equation}
Choose $w=\chi_{\textbf{B}}\in E_{\Phi_{1}}(\mathbb{R}^N)$, we have
\begin{equation*}
\begin{split}
\underset{\textbf{B}}{\int}\tilde{\xi}\dx=\underset{\mathbb{R}^N}{\int}\tilde{\xi}\chi_{\textbf{B}}\dx=\langle \xi,\chi_{\textbf{B}}\rangle\leq \Upsilon^{o}(u;\chi_{\textbf{B}})\leq\underset{\textbf{B}}{\int}\overline{f}(x,u(x))\dx.
\end{split}
\end{equation*}
This implies 
\begin{equation*}
\underset{\mathbb{R}^N}{\int}\tilde{\xi}\chi_{\textbf{B}}\dx\leq\underset{\textbf{B}}{\int}\overline{f}(x,u(x))\dx.
\end{equation*}
Hence, this contradicts \eqref{z1}. Therefore, 
\begin{equation*}
\tilde{\xi}\leq \overline{f}(x,u(x)), \text{ a.e. in } \mathbb{R}^N.
\end{equation*}
By similar arguments, one can also prove that 
\begin{equation*}
\tilde{\xi}\geq \underline{f}(x,u(x)), \text{ a.e. in }\mathbb{R}^N.
\end{equation*}
Hence 
\begin{equation*}
\tilde{\xi}(x)\in[\underline{f}(x,u(x)),\overline{f}(x,u(x))] \text{ a.e. in } \mathbb{R}^N.
\end{equation*}
Due to \cite[Theorem 2.2]{chang1981variational}, we have
\begin{equation*}
\partial\Upsilon|_{\mathbf{X}}(u)\subset\partial\Upsilon(u), ~\forall ~u\in \mathbf{X}.
\end{equation*} Hence, the proof is complete.\end{proof}
\section{Existence of first solution} \label{section 4}
In this section, we will establish the first solution of \eqref{main problem} via the Ekeland Variational Principle. The energy functional corresponding to \eqref{main problem} is $I_{\epsilon}: \mathbf{X}\rightarrow\mathbb{R}$ given by,
\begin{equation}\label{4.1}
I_\epsilon(u)=\frac{1}{p}\underset {\mathbb{R}^N}{\int}(|\nabla u|^p+V(x)|u|^p)\dx+\frac{1}{N}\underset {\mathbb{R}^N}{\int}(|\nabla u|^N+V(x)|u|^N)\dx-\underset {\mathbb{R}^N}{\int}F(x,u)\dx-\epsilon\underset {\mathbb{R}^N}{\int}gu \dx.
\end{equation}
 \begin{definition}[Weak Solution]\label{ws}
 We say $u \in \mathbf{X} $ is a (weak) solution for the problem if there is $\rho\in L^{\tilde{\Phi}_{1}}(\mathbb{R}^N)$ such that the following conditions are satisfied:
 \begin{enumerate}[label={($\bf ws{\arabic*}$)}]
\setcounter{enumi}{0}
 \item \label{solution1} For all $v\in \mathbf{X}$ \begin{align*} \underset {\mathbb{R}^N}{\int}(|\nabla u|^{p-2}\nabla u.\nabla v+V(x)|u|^{p-2}uv)\dx & + \underset {\mathbb{R}^N}{\int}(|\nabla u|^{N-2}\nabla u.\nabla v+V(x)|u|^{N-2}uv)\dx \\ 
    &-\underset {\mathbb{R}^N}{\int}\rho v\dx-\epsilon \underset {\mathbb{R}^N}{\int}g v\dx=0.\end{align*} 
 \item \label{solution 2} {$\rho(x)\in  \partial_tF(x,u(x))\text{ a.e. in } \mathbb{R}^N$}.
  \item \label{solution 3} { $m([ u > t_0]) >0.$}
  
  \end{enumerate}
  \end{definition}
  \begin{remark}
      Here, the $t_0$ used in \ref{solution 3} is the same as in assumption \ref{f1}.
  \end{remark}
 \noindent Let $I_{\epsilon}=J_{\epsilon}-\Upsilon$, where $\Upsilon$ is a functional defined in Theorem \ref{Theorem 3.1} and $ J_{\epsilon}:\mathbf{X}\rightarrow\mathbb{R}$ given by
 \begin{equation*}
 J_{\epsilon}(u)=\frac{1}{p}\underset {\mathbb{R}^N}{\int}(|\nabla u|^p+V(x)|u|^p)\dx+\frac{1}{N}\underset {\mathbb{R}^N}{\int}(|\nabla u|^N+V(x)|u|^N)\dx-\epsilon\underset {\mathbb{R}^N}{\int}gu \dx.
 \end{equation*}
 By similar arguments used in \cite{badiale2010semilinear}, it can be proved that $J_{\epsilon}\in C^{1}(\mathbf{X},\mathbb{R})$ and by Theorem \ref{Theorem 3.1} and Lemma \ref{2.6},  $\Upsilon\in Lip_{loc}(\mathbf{X},\mathbb{R})$. Hence, $I_{\epsilon}\in Lip_{loc}(\mathbf{X},\mathbb{R}).$ So, $\partial I_{\epsilon}(u)=J_{\epsilon}'(u)-\partial\Upsilon(u)$, for all $u\in\mathbf{X}.$
 \begin{lemma}\label{le4.1}
 Let the assumptions \ref{f2} and \ref{f6} hold. Then there are $\kappa, \epsilon_0, \rho, \delta^* >0$, such that 
 \begin{equation}\label{4.1.1}
 I_{\epsilon}(u)\geq \kappa,\ \forall~\|u\|_{\mathbf{X}}=\rho
 \end{equation}
 and 
 \begin{equation}\label{4.1.2}
 c_{\epsilon}=\underset{\|u\|\leq\rho}{\inf}I_{\epsilon}(u)<-\delta^*, \forall~\epsilon \in (0,\epsilon_0].
 \end{equation}
 \end{lemma}
\begin{proof} By the assumption \ref{f2}, for any $\upsilon \in (0,\lambda_1)$, $\alpha>\alpha_0$ close to $\alpha_0$, we have 
 \begin{equation}\label{4.2}
 |\xi|\leq \frac{(\lambda_1-\upsilon)|t|^{N-1}}{N},\forall~|t|\leq\Bar{\delta},\ x\in\mathbb{R}^N, \text{ and } \xi \in \partial_tF(x,t),
 \end{equation}
 for some ~$\Bar{\delta}>0.$
 By assumption \ref{f6}, for $\alpha>\alpha_0$ close to $\alpha_0$ and $q>N$ one has
 \begin{equation}\label{4.2.1}
 |\xi|\leq C_{\bar{\delta}}|t|^{q-1}\Phi(\alpha|t|^\frac{N}{N-1}),\forall~|t|\geq \Bar{\delta} \text{ and } x\in\mathbb{R}^N. 
  \end{equation}
  For $q>N$, on combining \eqref{4.2} and \eqref{4.2.1}, we have
   \begin{equation*}
   |F(x,t)|\leq \frac{(\lambda_1-\upsilon)|t|^{N}}{N}+ C_{\bar{\delta}}|t|^{q}\Phi(\alpha|t|^\frac{N}{N-1}),~\forall~t\in\mathbb{R} \text{ and } x\in\mathbb{R}^N.
   \end{equation*}
   The above also follows by applying Lebourg’s Mean-Value theorem.
   Let $\delta\in(0,1)$ very small and there holds $0<\|u\|_{\mathbf{X}}\leq\delta$. {Choose $\alpha>\alpha_0$ close to $\alpha_0$ such that $\alpha\|u\|_{\mathbf{X}}^{\frac{N}{N-1}}<\alpha_N.$} Then, using Lemmas \ref{2.2} and \ref{Lemma 2.3},  we deduce
\begin{align*}
I_{\epsilon}(u)&\geq \frac{\left(\|u\|_{W_V^{1,p}}^p+\|u\|_{W_V^{1,N}}^N\right)}{N}-\frac{(\lambda_1-\upsilon)}{N}\|u\|_{L^N}^N-C_{\Bar{\delta}}\|u\|_{\mathbf{X}}^q-\epsilon\|g\|_{*}\|u\|_{\mathbf{X}}\\
&\geq\frac{\left(\|u\|_{W_V^{1,p}}^N+\|u\|_{W_V^{1,N}}^N\right)}{N}-\frac{(\lambda_1-\upsilon)}{N}\|u\|_{L^N}^N-C_{\Bar{\delta}}\|u\|_{\mathbf{X}}^q-\epsilon\|g\|_{*}\|u\|_{\mathbf{X}}\\
&\geq\frac{2^{1-N}}{N} \left(\|u\|_{W_V^{1,p}}+\|u\|_{W_V^{1,N}}\right)^N-\frac{(\lambda_1-\upsilon)}{N}\|u\|_{L^N}^N-C_{\Bar{\delta}}\|u\|_{\mathbf{X}}^q-\epsilon\|g\|_{*}\|u\|_{\mathbf{X}}\\
&\geq\frac{2^{1-N}}{N} \left(\|u\|_{W_V^{1,p}}+\|u\|_{W_V^{1,N}}\right)^N- \frac{1}{N}\frac{(\lambda_1-\upsilon)}{\lambda_1}\|u\|_{\mathbf{X}}^N-C_{\Bar{\delta}}\|u\|_{\mathbf{X}}^q-\epsilon\|g\|_{*}\|u\|_{\mathbf{X}}\\
&=\frac{1}{N}\left(2^{1-N}-\frac{(\lambda_1-\upsilon)}{\lambda_1}\right)\|u\|_{\mathbf{X}}^N-C_{\Bar{\delta}}\|u\|_{\mathbf{X}}^q-\epsilon\|g\|_{*}\|u\|_{\mathbf{X}}.
\end{align*}  
   Now choose $\upsilon=(1-\frac{1}{2^N})\lambda_1$, thus, we obtain
   \begin{equation}\label{4.5}
   I_{\epsilon}(u)\geq \frac{1}{N2^N}\|u\|_{\mathbf{X}}^N-C_{\bar{\delta}}\|u\|_{\mathbf{X}}^q-\epsilon\|g\|_{*}\|u\|_{\mathbf{X}}.
   \end{equation}
   {Now from elementary calculus it follows that function $h(l)=\frac{l^N}{2^{N+1}N}-C_{\bar{\delta}}l^q$, where $l\in[0,\delta]$ attains its non zero maximum $\kappa$ at some  $\rho\in(0,\delta].$} Now for $u \in \mathbf{X}$, from \eqref{4.5} and for $\|u\|_\mathbf{X}=\rho$, we have
   \begin{equation*}
   I(u)\geq \frac{1}{N2^N}\rho^N-C_{\bar{\delta}}\rho^q-\epsilon\|g\|_{*}\rho\geq\frac{\rho^N}{2^{N+1}N}-C_{\bar{\delta}}\rho^q=h(\rho)=\kappa,
   \end{equation*}
   for all $\epsilon \in (0,\epsilon_0],$ where $\epsilon_0=\frac{\rho^{N-1}}{2^{N+1}N\|g\|_*}.$ Hence \eqref{4.1.1} follows. Now consider $\tilde{w}\in\mathbf{X}$ such that $\|\tilde{w}\|_\mathbf{X}=1$ and $\underset{\mathbb{R}^N}{\int}g\tilde{w}\dx>0$. For each $m>0$ from \eqref{4.1}, we get
   \begin{equation*}
   I_{\epsilon}(m\tilde{w})\leq \frac{m^p}{p}+\frac{m^N}{N}-\epsilon m\underset{\mathbb{R}^N}{\int} g\tilde{w}\dx.
   \end{equation*}
   For $m\in(0,\rho)$, we have
   \begin{equation}\label{4.10}
   I_{\epsilon}(m\tilde{w})\leq \frac{1}{p}\left(m^p+\frac{pm^N}{N}\right)-\epsilon m \underset{\mathbb{R}^N}{\int} g\tilde{w} \dx\leq\frac{2m^p}{p}-\epsilon m\underset{\mathbb{R}^N}{\int} g\tilde{w} \dx.
   \end{equation}
   Now by fixing $m=m(\epsilon)>0 $ small enough such that 
$$\delta^*=-\frac{2m^p}{p}+\epsilon m\underset{\mathbb{R}^N}{\int} g\tilde{w} \dx>0.$$
 From \eqref{4.10}, for  $\|m\tilde{w}\|_\mathbf{X}<\rho$ ,we have $$I_{\epsilon}(m\tilde{w})<-\delta^*<0.$$
So
\begin{equation*}
c_{\epsilon}=\underset{\|u\|\leq\rho}{\inf}I_{\epsilon}(u)<-\delta^*, ~\forall~\epsilon \in (0,\epsilon_0].
\end{equation*}
Hence, \eqref{4.1.2} follows, and the proof is complete. \end{proof}
We will prove some technical lemmas before proving Theorem \ref{main result 1}.
 \begin{lemma}\label{lemma4.2}
 Assume \ref{f2} and \ref{f6} hold. Then, there is a sequence $\{u_n\}\subset \overline{B}_{\rho}(0)\subset\mathbf{X}$ such that 
 \begin{itemize}
     \item[(i)] $I_{\epsilon}(u_n)\rightarrow c_{\epsilon}$ \text{ as } $n\rightarrow+\infty$,
     \item[(ii)] $\lambda_{\epsilon}(u_n)\rightarrow 0$ \text{ as } $n\rightarrow+\infty$,
 \end{itemize} 
{ where $\lambda_{\epsilon}(u_n)= \displaystyle\min_{w \in \partial I_{\epsilon}(u_n)} \|w\|_{\mathbf{X}}^* $ and $\partial I_{\epsilon}(u_n)$ is the generalized gradient defined in Definition \ref{def2.2}.}\end{lemma}
 \begin{proof}
 First of all, we have $I_{\epsilon}\in Lip_{loc}(\mathbf{X},\mathbb{R})$ (due to Lemma \ref{2.6} and Theorem \ref{Theorem 3.1}). Thanks to Lemma \ref{le4.1}, since  $ c_{\epsilon}=\underset{\|u\|\leq\rho}{\inf}I_{\epsilon}(u)<-\delta^*<0$ and $I_{\epsilon}(u)\geq \kappa \text{ for } \|u\|=\rho$, we get
 $$\underset{\partial B_{\rho}}{\inf} I_{\epsilon}(u)-\underset{\overline{B}_{\rho}}{\inf} I_{\epsilon}(u)>0.$$
 Choose a sufficiently large $n$ such that 
 \begin{equation}\label{4.12}
 \frac{1}{n} \in (0, \underset{\partial B_{\rho}}{\inf} I_{\epsilon}(u)-\underset{\overline{B}_{\rho}}{\inf} I_{\epsilon}(u)) .
 \end{equation}
 Now by applying Ekeland Variational Principle on $I_{\epsilon}:\overline{B}_{\rho}\rightarrow \mathbb{R}$, we obtain a sequence $\{u_n\}\subset\overline{B}_{\rho}$ such that
 \begin{align}\label{4.13}
 \begin{split}c_{\epsilon}\leq I_{\epsilon}(u_n)&\leq c_{\epsilon}+\frac{1}{n}, \text{ and } \\ 
I_{\epsilon}(u_n)-I_{\epsilon}(u)&\leq \frac{\|u_n-u\|_\mathbf{X}}{n},\forall~u\in \overline{B}_{\rho} \text{ and } u_n\neq u. \end{split}
 \end{align}
 From \eqref{4.13}, the proof of part (i) follows as $n\rightarrow+\infty$ and \eqref{4.12} implies  
 \begin{equation}\label{4.15}
 I_{\epsilon}(u_n)\leq \underset{\overline{B}_{\rho}}{\inf}I_{\epsilon}(u)+\frac{1}{n}<\underset{\partial B_{\rho}}{\inf} I_{\epsilon}(u) .
\end{equation}
{Thus, \eqref{4.15} implies that $u_n\in B_{\rho}.$ Let $\sigma>0$ be small enough such that $u_{\sigma}=u_n+\sigma v \in B_{\rho}$, for $v\in \mathbf{X}.$ Replace $u=u_\sigma$ in \eqref{4.13}, we get
\begin{equation*}
I_{\epsilon}(u_n+\sigma v)-I_{\epsilon}(u_n)+\frac{1}{n}\|u_n-u\|_\mathbf{X}\geq 0, \end{equation*}
which further implies
$$I_{\epsilon}(u_n+\sigma v)-I_{\epsilon}(u_n)+\frac{\sigma}{n}\|v\|_\mathbf{X}\geq 0.$$}Further, we deduce
\begin{equation*}
-\frac{1}{n}\|v\|_{\mathbf{X}}\leq \underset{\sigma\downarrow 0}{{\limsup}}\frac{I_{\epsilon}(u_n+\sigma v)-I_{\epsilon}(u_n)}{\sigma}\leq I_{\epsilon}^0(u_n;v)=\underset{w\in\partial I_{\epsilon}(u_n)}{\max}\langle w,v \rangle.\end{equation*}
Now, we replace $v$ by $-v$ in the above inequality we obtained, and we get
$$ -\frac{1}{n}\|v\|_{\mathbf{X}}\leq -\underset{w\in\partial I_{\epsilon}(u_n)}{\min}\langle w,v \rangle,\forall~v\in \mathbf{X}.$$ It gives us
$$\underset{w\in\partial I_{\epsilon}(u_n)}{\min}\langle w,v \rangle\leq \frac{1}{n}\|v\|_{\mathbf{X}}.$$Therefore, 
\begin{equation*}
\underset{\|v\|_\mathbf{X}=1}{\sup}~~\underset{w\in\partial I_{\epsilon}(u_n)}{\min}\langle w,v \rangle\leq \frac{1}{n}.
\end{equation*}
By the minimax argument in \cite{MR2424292}, we have
\begin{equation*}
\underset{w\in\partial I_{\epsilon}(u_n)}{\min}\underset{\|v\|_\mathbf{X}=1}{\sup}\langle w,v \rangle\leq \frac{1}{n}.
\end{equation*}
Hence, the proof of part (ii) is complete. \end{proof}
Since there is no Hilbert structure on $\mathbf{X}$, we need a careful analysis of gradient convergence. The next will serve our purpose. 
 \begin{lemma}\label{l4.3}
Let the assumptions \ref{f2} and \ref{f6} hold. Consider  $\{u_n\}$ be a sequence obtained in Lemma \ref{lemma4.2} and $u_{\epsilon} \in \mathbf{X}$. If $$\underset{n\rightarrow+\infty}{\limsup}\|u_n\|_{W^{1,N}}^\frac{N}{N-1}<\frac{\alpha_N}{\alpha_0} \text{ and } u_n\rightharpoonup  u_{\epsilon} \text{ in }
 \mathbf{X},$$ then, $\nabla u_n\rightarrow \nabla u_{\epsilon}$ a.e. in $\mathbb{R}^N.$
 \end{lemma}
 \begin{proof} Since $\{u_n\}\subset\mathbf{X}$ is a bounded sequence in $\mathbf{X}$, so up to a subsequence still denoted by $\{u_n\}$, we have
$$\begin{cases}
 u_n\rightharpoonup u_\epsilon \text{ in } \mathbf{X},\\
 u_n\rightarrow u_\epsilon \text{ a.e. in } \mathbb{R}^N,\\
 u_n\rightharpoonup u_\epsilon \text{ in }W_{V}^{1,t}(\mathbb{R}^N),~t\in\{p,N\},\\
 u_n\rightarrow u_{\epsilon} \text{ in } L^q(\mathbb{R}^N),\forall~q\in[1,+\infty).
 \end{cases}$$
Now we make the following claim:\\
{{\bf Claim:} For $\{u_n\} \subset \mathbf{X}$, we have $$\nabla u_n\rightarrow \nabla u_\epsilon \text{ in } L^{N}(\mathbb{R}^N)$$ if and only if 
 $$\underset{n\rightarrow+\infty}{\lim}\underset{\mathbb{R}^N}{\int}\left( |\nabla u_n(x)|^{N-2}\nabla u_n(x)-|\nabla u_\epsilon(x)|^{N-2}\nabla u_\epsilon(x)\right)\cdot\left(\nabla u_n(x)-\nabla(u_\epsilon(x)\right)\dx=0.$$}
{ We follow the proof of \cite[Lemma 2.1]{MR4587597} to establish the above claim} as the  function $a(x,\nabla u_n)=|\nabla u_n|^{N-2}\nabla u_n$ satisfy all the hypothesis of cited lemma.  By the convexity of $s\mapsto|s|^N$, it follows that
 $$(|\xi|^{N-2}\xi-|\eta|^{N-2}\eta)\cdot(\xi-\eta)\geq 0.$$
 Next, we set $$P_n(x)=\left( |\nabla u_n(x)|^{N-2}\nabla u_n(x)-|\nabla u_\epsilon(x)|^{N-2}\nabla u_\epsilon(x)\right)\cdot\left(\nabla u_n(x)-\nabla u_\epsilon(x)\right).$$
 Consider $\psi\in C_c^\infty(\mathbb{R}^N)$ with $0\leq\psi\leq1$, $\psi\equiv 1$ on $B_{R}(0)$ and $\psi\equiv0$ on $B_{2R}^c(0)$. From Lemma \ref{lemma4.2} (ii), it follows that, there exists $w_n \in~ \partial I_\epsilon (u_n)$ such that
 \begin{equation}\label{4.16}
 \begin{split}
 \langle w_n, \psi(u_n-u_\epsilon)\rangle =o_n(1)\quad\text{as}\quad n\to\infty.
 \end{split}
 \end{equation}
 By using $P_n(x)\geq 0$, we have
 \begin{equation*}
 0\leq \underset{B_{R}(0)}{\int} P_n(x)\dx\leq \underset{\mathbb{R}^N}{\int} P_n(x)\psi(x)\dx.
 \end{equation*}
Notice that $\nabla ((u_n-u_\epsilon)\psi)=(\nabla u_n-\nabla u_\epsilon)\psi + (u_n-u_\epsilon)\nabla\psi$. It follows that
{
\begin{align}\label{4.18}
\underset{B_{R}(0)}{\int} P_n(x)\dx&\notag\leq \underset{\mathbb{R}^N}{\int}\left( |\nabla u_n|^{N-2}\nabla u_n-|\nabla u_\epsilon|^{N-2}\nabla u_\epsilon\right)\cdot\left(\nabla u_n-\nabla u_\epsilon\right)\psi \dx\\ \notag
&= \underset{\mathbb{R}^N}{\int}\left( |\nabla u_n|^{N-2}\nabla u_n-|\nabla u_\epsilon|^{N-2}\nabla u_\epsilon\right)\cdot\left(\nabla ((u_n-u_\epsilon)\psi)-(u_n-u_\epsilon)\nabla\psi \right) \dx\\ \notag
 &=\underset{\mathbb{R}^N}{\int}\left( |\nabla u_n|^{N-2}\nabla u_n\right)\cdot\left(\nabla ((u_n-u_\epsilon)\psi)\right)\dx\\ \notag  &\qquad-\underset{\mathbb{R}^N}{\int}\left( |\nabla u_n|^{N-2}\nabla u_n\right)\cdot\left((u_n-u_\epsilon)\nabla\psi \right)\dx\\ 
  &\qquad-\underset{\mathbb{R}^N}{\int}
 \left(|\nabla u_\epsilon|^{N-2}\nabla u_\epsilon\right)\cdot\left( (\nabla u_n-\nabla u_\epsilon)\psi\right)\dx.
\end{align}
} 
 Next, we define 
 \begin{equation*}
 \begin{split}
 \varrho_n(x)=\left( |\nabla u_n(x)|^{N-2}\nabla u_n(x)-|\nabla u_\epsilon(x)|^{N-2}\nabla u_\epsilon(x)\right)\cdot\left(\nabla u_n(x)-\nabla u_\epsilon(x)\right)\\+\left( |\nabla u_n(x)|^{p-2}\nabla u_n(x)-|\nabla u_\epsilon(x)|^{p-2}\nabla u_\epsilon(x)\right)\cdot\left(\nabla u_n(x)-\nabla u_\epsilon(x)\right).
\end{split}
 \end{equation*}
 Due to convexity of map $s\mapsto|s|^N$; we have  $\varrho_n(x)\geq P_n(x)\geq 0.$ This implies
\begin{equation*}
0\leq \underset{B_{R}(0)}{\int} P_n(x)\dx\leq \underset{\mathbb{R}^N}{\int} \varrho_n(x)\psi(x)\dx.
\end{equation*}
From \eqref{4.16} and \eqref{4.18}, it easy to see that
\begin{align*}
\underset{B_{R}(0)}{\int} P_n(x)&\dx \notag \leq o_n(1)-\underset{\mathbb{R}^N}{\int}|\nabla u_n|^{p-2}\nabla u_n.\nabla\psi (u_n-u_\epsilon)\dx \\ \notag
&-\underset{\mathbb{R}^N}{\int}\left( |\nabla u_n|^{N-2}\nabla u_n\right).\left((u_n-u_\epsilon)\nabla\psi \right)\dx -\underset{\mathbb{R}^N}{\int}
 \left(|\nabla u_\epsilon|^{p-2}\nabla u_\epsilon\right)\cdot\left( (\nabla u_n-\nabla u_\epsilon)\psi\right)\dx\\ \notag  &-\underset{\mathbb{R}^N}{\int}
 \left(|\nabla u_\epsilon|^{N-2}\nabla u_\epsilon\right)\cdot\left( (\nabla u_n-\nabla u_\epsilon)\psi\right)\dx -\underset{\mathbb{R}^N}{\int} V_0|u_n|^{p-2}u_n(u_n-u_{\epsilon})\psi \dx\\ \notag  &-\underset{\mathbb{R}^N}{\int} V_0|u_n|^{N-2}u_n(u_n-u_{\epsilon})\psi \dx 
 +\epsilon \underset{\mathbb{R}^N}{\int} g\psi (u_n-u_{\epsilon})\dx\\ &+\underset{\mathbb{R}^N}{\int} \rho_n \psi(u_n-u_\epsilon)\dx.
\end{align*}
For simplicity, we denote the  above expression as follows
\begin{equation*}
\underset{B_{R}(0)}{\int} P_n(x)\dx\leq o_n(1)-\Psi_1-\Psi_2-\Psi_3-\Psi_4-\Psi_5-\Psi_6+\Psi_7+\Psi_8.
\end{equation*}
Now we will show that each~ $\Psi_i \rightarrow 0$ as $n\rightarrow+\infty$ for each $1\leq i\leq 8.$
Consider $$\Psi_1=\underset{\mathbb{R}^N}{\int}|\nabla u_n|^{p-2}\nabla u_n.\nabla\psi (u_n-u_\epsilon)\dx.$$ First, we estimate
\begin{equation*} 
\begin{split}
\lvert{\Psi_1}\rvert=\left|{\underset{\mathbb{R}^N}{\int}|\nabla u_n|^{p-2}\nabla u_n.\nabla\psi (u_n-u_\epsilon)\dx}\right|\leq\|\nabla \psi\|_{\infty}{\underset{\mathbb{R}^N}{\int}|\nabla u_n|^{p-1}  |(u_n-u_\epsilon)|\dx}.
\end{split}
\end{equation*}
By H\"older's inequality and $\mathbf{X}\hookrightarrow\hookrightarrow L^q(\mathbb{R}^N),~q\geq 1$
\begin{equation*}
\lvert{\Psi_1}\rvert \leq\|\nabla \psi\|_{\infty}\left(\underset{\mathbb{R}^N}{\int}|\nabla u_n|^p\dx\right)^\frac{p-1}{p}\left(\underset{\mathbb{R}^N}{\int}|u_n-u_\epsilon|^p\dx\right)^\frac{1}{p} \rightarrow 0 \text{ as } n\rightarrow+\infty.
\end{equation*}
So $\Psi_1\rightarrow 0$, similarly, we can prove that $\Psi_2,\Psi_5,\Psi_6\rightarrow 0$ as $n\rightarrow+\infty.$
Now we consider $$\Psi_3=\underset{\mathbb{R}^N}{\int}
 \left(|\nabla u_\epsilon|^{p-2}\nabla u_\epsilon\right)\cdot\left( (\nabla u_n-\nabla u_\epsilon)\psi\right)\dx.$$ Suppose,
 \begin{equation*}
  \langle F_1, v\rangle = \underset{\mathbb{R}^N}{\int} |\nabla u_\epsilon|^{p-2}\nabla u_{\epsilon}.\nabla v \psi \dx,~~v\in W^{1,p}(\mathbb{R}^N).
 \end{equation*}
 H\"older's inequality ensures that, $F_1\in (W^{1,p}(\mathbb{R}^N)^*$. Since $u_n\rightharpoonup u_\epsilon$ in $ \mathbf{X}$, we deduce that $\Psi_3\rightarrow 0$ as $n\rightarrow+\infty$. By similar arguments $\Psi_4\rightarrow 0$ as $n\rightarrow+\infty.$ As $g\in \mathbf{X}^*$ and $u_n\rightharpoonup u_\epsilon$ in $\mathbf{X}$, so $$\underset{\mathbb{R}^N}{\int} g\psi (u_n-u_{\epsilon})\dx \rightarrow 0, ~~n\rightarrow +\infty.$$
 Now consider $$\Psi_8=\underset{\mathbb{R}^N}{\int} \rho_n \psi(u_n-u_\epsilon)\dx.$$ By assumption \ref{f6}, we have
 \begin{equation*}
 \begin{split}
 \left|{\underset{\mathbb{R}^N}{\int} \rho_n \psi(u_n-u_\epsilon)\dx}\right|\leq c_1 \underset{\mathbb{R}^N}{\int}|u_n|^{N-1}|\psi\|(u_n-u_\epsilon)|\dx+c_2 \underset{\mathbb{R}^N}{\int}\Phi(\alpha_0|u_n|^\frac{N}{N-1})|\psi\|u_n-u_\epsilon|\dx.
 \end{split}
 \end{equation*}
 Let us define $$\Psi_{8}^{1}=\underset{\mathbb{R}^N}{\int}|u_n|^{N-1}|\psi\|(u_n-u_\epsilon)|\dx \text{ and } \Psi_{8}^2=\underset{\mathbb{R}^N}{\int}\Phi(\alpha_0|u_n|^\frac{N}{N-1})|\psi\|u_n-u_\epsilon|\dx.$$
 Using the previous arguments, one can prove that $\Psi_{8}^1\rightarrow 0$ as $n\rightarrow+\infty.$ Now, since $$\underset{n\rightarrow+\infty}\limsup\|u_n\|_{W^{1,N}}^\frac{N}{N-1}<\frac{\alpha_N}{\alpha_0}, $$ by definition, there exists, $l>0$ such that $$ \|u_n\|_{W^{1,N}}^\frac{N}{N-1}<l<\frac{\alpha_N}{\alpha_0}, \forall  n\geq n_0.$$ Choose $\alpha>\alpha_0$ (sufficiently close to $\alpha_0$) and $r>1$ but close to $1$ such that $r\alpha\|u_n\|_{W^{1,N}}^\frac{N}{N-1}<l<\alpha_N$ for all $n\geq n_0$ and $\frac{1}{r}+\frac{1}{r'}=1$, then apply H\"older's inequality and using Lemma \ref{2.4}, we have
 \begin{align*}
 |\Psi_8^2| \leq &\left|\underset{\mathbb{R}^N}{\int}\Phi(\alpha|u_n|^\frac{N}{N-1})|u_n-u_\epsilon|\dx\right|\\
&\leq \left(\underset{\mathbb{R}^N}{\int}\left(\Phi\left(\alpha|u_n|^\frac{N}{N-1}\right)\right)^r\dx\right)^\frac{1}{r} \left(\underset{\mathbb{R}^N}{\int}\left(|u_n-u_\epsilon|^r{'}\dx\right)\right)^\frac{1}{r'}\\
&\leq \left(\underset{\mathbb{R}^N}{\int}\Phi\left(r\alpha|u_n|^\frac{N}{N-1}\right)dx\right)^\frac{1}{r} \left(\underset{\mathbb{R}^N}{\int}|u_n-u_\epsilon|^r{'}\dx\right)^\frac{1}{r'}\\
&\leq \left(\underset{\mathbb{R}^N}{\int}\Phi\left(\frac{r\|u_n\|_{W^{1,N}}^\frac{N}{N-1}\alpha|u_n|^\frac{N}{N-1}}{\|u_n\|_{W^{1,N}}^\frac{N}{N-1}}\right)dx\right)^\frac{1}{r}  \left(\underset{\mathbb{R}^N}{\int}|u_n-u_\epsilon|^r{'}\dx\right)^\frac{1}{r'}\\
&\leq C\left(\underset{\mathbb{R}^N}{\int}|u_n-u_\epsilon|^r{'}\dx\right)^\frac{1}{r'} \rightarrow 0 \text{ as } n\rightarrow +\infty.
\end{align*}
 { Hence, $\Psi_8\rightarrow 0 \text{ as }n\rightarrow +\infty.$ This implies
 \begin{equation*}
 \underset{n\rightarrow +\infty}{\lim}\underset{B_{R}(0)}{\int} P_n(x)\dx=0.
 \end{equation*}
 Therefore, we obtain
 \begin{equation*}
 \nabla u_n\rightarrow \nabla u_\epsilon \text{ in } L^{N}({B_{R}(0)}).
\end{equation*}}Since it holds for all $R>0$, up to a subsequence (is still denoted by the same), we get
\begin{equation*}
\nabla u_n\rightarrow \nabla u_\epsilon \text{ a.e. in } \mathbb{R}^N.
\end{equation*}
\end{proof}
\begin{lemma}\label{lemma4.4}
Suppose \ref{f6} holds. For $\rho_n\in \partial_tF(x,u_n)$, there is $\rho_0 \in L^{\tilde{\Phi}_{1}}(\mathbb{R}^N)$, such that 
\begin{equation*}
\underset{\mathbb{R}^N}{\int}\rho_n v\dx \to \underset{\mathbb{R}^N}{\int}\rho_0 v\dx,~\forall~ v\in \mathbf{X}.
\end{equation*}
\end{lemma}
\begin{proof} First, we will show that $\{\rho_n\}$ is bounded in $L^{\tilde{\Phi}_{1}}(\mathbb{R}^N)$. Since $\Phi_1$ is an increasing function and from assumption \ref{f6} there is $\alpha>\alpha_0$ close to it  such that $ \alpha\rho^\frac{N}{N-1}<\alpha_N$ and $q>N$ such that, we have
\begin{equation*}
\underset{\mathbb{R}^N}{\int}\tilde{\Phi}_{1}(\rho_n)\dx\leq \underset{\mathbb{R}^N}{\int}\tilde{\Phi}_{1}\left(c_1|u_n|^{N-1}+c_2|u_n|^q\Phi(\alpha|u|^\frac{N}{N-1})\right)\dx.
\end{equation*}
Using the $\Delta_2$-condition and convexity of $\tilde{\Phi}_{1}$, we have
\begin{equation*}
\underset{\mathbb{R}^N}{\int}\tilde{\Phi}_{1}(\rho_n)\dx\leq \frac{1}{2}\underset{\mathbb{R}^N}{\int}\tilde{\Phi}_{1}(1)\zeta(2c_1|u_n|^{N-1})\dx+ \frac{\zeta(2c_2)}{2}\underset{\mathbb{R}^N}{\int}\tilde{\Phi}_{1}\left(|u_n|^q\Phi_{1}\left(\alpha^\frac{N-1}{N}|u|\right)\right)\dx.
\end{equation*}
{ Using \eqref{e5}, we get 
\begin{align*}
\underset{\mathbb{R}^N}{\int}\tilde{\Phi}_{1}(\rho_n)\dx &\leq C_1\left(\underset{\mathbb{R}^N}{\int}|u_n|^{N-1}\dx+\underset{\mathbb{R}^N}{\int}|u_n|^{N(N-1)}\dx\right)\\ & +C_2\left(\underset{\mathbb{R}^N}{\int}\left(|u_n|^{q+1}+|u_n|^{N(q+1)}\right)\Phi_{1}\left(\alpha^\frac{N-1}{N}|u|\right)\dx\right).
\end{align*}}
The fact that $\mathbf{X}\hookrightarrow L^q(\mathbb{R}^N),$ for all $q\geq 1$ and Lemma \ref{2.2} implies there is $\tilde{C}>0$ such that 
\begin{equation*}
\underset{\mathbb{R}^N}{\int}\tilde{\Phi}_{1}(\rho_n)\dx\leq \tilde{C},~\forall~n\in\mathbb{N}.
\end{equation*}
\text{Hence $\{\rho_n\}$ is bounded in $L^{\tilde{\Phi}_{1}}(\mathbb{R}^N)$}. So sequence of functionals $\tilde{\rho_n}\subset\partial\Upsilon(u_n)\subset(E_{\Phi_1}(\mathbb{R}^N))^*$ corresponding to $\{\rho_n\}$ is also bounded in $(E_{\Phi_1}(\mathbb{R}^N))^*$. So there is $\tilde{\rho_0}\in (E_{\Phi_1}(\mathbb{R}^N))^*$ such that $\tilde{\rho_n}\overset{\ast}{\rightharpoonup}\tilde{\rho_0}$ in $(E_{\Phi_1}(\mathbb{R}^N))^*$ i.e
\begin{equation}\label{4.22}
\underset{\mathbb{R}^N}{\int} \rho_n v\dx=\langle \tilde{\rho_n},v\rangle \rightarrow \langle \tilde{\rho_0},v\rangle =\underset{\mathbb{R}^N}{\int} \rho_0 v\dx,~\forall v\in\mathbf{X},
\end{equation}
for some $\rho_0\in L^{\tilde{\Phi}_{1}}(\mathbb{R}^N).$ Hence, proof of the Lemma is complete. \end{proof}
Now we are in a situation to include the proof of Theorem \ref{main result 1}.
\begin{proof}[Proof of Theorem \ref{main result 1}]
First, we set $\hat\epsilon=\min\{\epsilon_0,\tilde\epsilon\},$
where
\begin{equation}\label{4.271}
\tilde\epsilon=\left(\min\{1,V_0\}\frac{1}{\rho \|g\|_*}\left(\frac{\beta(N-\beta)}{N\beta(\beta-1)}\right)\left(\frac{\alpha_N}{\alpha_0}\right)^\frac{N}{N-1}\right)
\end{equation}
and $\epsilon_0$ is obtained in Lemma \ref{4.1}.
From Lemma \ref{lemma4.2}, we obtain a sequence $\{u_n\}\subset \overline{B}_{\rho}(0)\subset\mathbf{X}$ such that 
 \begin{itemize}
     \item[(i)] $I_{\epsilon}(u_n)\rightarrow c_{\epsilon}$ \text{ as } $n\rightarrow+\infty$,
     \item[(ii)] $\lambda_{\epsilon}(u_n)\rightarrow 0$ \text{ as } $n\rightarrow+\infty$.
\end{itemize}
Since $\{u_n\}$ is a bounded sequence in $\mathbf{X}$ and it is reflexive, up to a subsequence still denoted by itself $$u_n\rightharpoonup u_\epsilon \text{ in } \mathbf{X}.$$
Let $\beta>N$ assumed in \ref{f4} and due to Lemma \ref{4.1}, we have
\begin{equation*}
0>c_{\epsilon}=I_{\epsilon}(u_n)-\frac{1}{\beta} \langle w_n, u_n \rangle + o_n(1).
\end{equation*}
Then, the above implies
\begin{equation*}
0>\left(\frac{1}{p}-\frac{1}{\beta}\right)\|u_n\|_{W_{V}^{1,p}}^p+ \left(\frac{1}{N}-\frac{1}{\beta}\right)\|u_n\|_{W_{V}^{1,N}}^N-\epsilon \frac{\beta -1}{\beta}\|g\|_{*}\rho+ o_n(1).
\end{equation*} Further, we deduce
\begin{equation}\label{4.24}
\left(\frac{1}{N}-\frac{1}{\beta}\right)\left(\|u_n\|_{W_{V}^{1,p}}^p+\|u_n\|_{W_{V}^{1,N}}^N\right)<\epsilon \frac{\beta -1}{\beta}\|g\|_{*}\rho+ o_n(1).
\end{equation}
Since $W_V ^{1,N}(\mathbb{R}^N)\hookrightarrow W ^{1,N}(\mathbb{R}^N)$, we obtain 
\begin{equation}\label{4.2601}
\left(\frac{1}{N}-\frac{1}{\beta}\right) \min\{1,V_0\}\|u_n\|_{W^{1,N}}^N\leq \left(\frac{1}{N}-\frac{1}{\beta}\right)\left(\|u_n\|_{W_{V}^{1,p}}^p+\|u_n\|_{W_{V}^{1,N}}^N\right).
\end{equation}
Then from \eqref{4.24} and \eqref{4.2601}, we get
\begin{equation*}
\left(\frac{1}{N}-\frac{1}{\beta}\right) \min\{1,V_0\}\|u_n\|_{W^{1,N}}^N\leq \left(\frac{1}{N}-\frac{1}{\beta}\right)\left(\|u_n\|_{W_{V}^{1,p}}^p+\|u_n\|_{W_{V}^{1,N}}^N\right)<\epsilon \frac{\beta -1}{\beta}\|g\|_{*}\rho+ o_n(1).
\end{equation*} 
Hence, from the above, we get
\begin{equation*}\label{4.272}
\|u_n\|_{W^{1,N}}^N<\epsilon \frac{\beta -1}{\beta}\|g\|_{*}\rho\frac{N\beta}{(N-\beta)}\frac{1}{\min\{1,V_0\}}+o_n(1).
\end{equation*}
From \eqref{4.271} and \eqref{4.24}, we obtain
\begin{equation*}
\underset{n\rightarrow+\infty}\limsup\|u_n\|_{W^{1,N}}^\frac{N}{N-1}<\frac{\alpha_N}{\alpha_0} \text{ as } n\rightarrow+\infty.
\end{equation*}
So by Lemma \ref{l4.3}, we deduce $$\nabla u_n(x)\rightarrow \nabla u_{\epsilon}(x) \text{ a.e. in } \mathbb{R}^N.$$
{Since $\{|\nabla u_n|^{p-2}\nabla u_n\}$  is bounded in $L^{\frac{p}{p-1}}(\mathbb{R}^N)$, we get
\begin{align*}
|\nabla u_n|^{p-2}\nabla u_n\rightharpoonup |\nabla u_\epsilon|^{p-2}\nabla u_\epsilon \text{ in } L^{\frac{p}{p-1}}(\mathbb{R}^N).
\end{align*}
Now for $\psi \in C_c^{\infty}(\mathbb{R}^N)$, we have
\begin{equation*}
\underset{n\rightarrow+\infty}{\lim}\underset{\mathbb{R}^N}{\int}|\nabla u_n|^{p-2}\nabla u_n \cdot\nabla \psi \dx =\underset{n\rightarrow+\infty}{\lim}\underset{B_R}{\int}|\nabla u_n|^{p-2}\nabla u_n\cdot \nabla \psi \dx = \underset{\mathbb{R}^N}{\int}|\nabla u_\epsilon|^{p-2}\nabla u_\epsilon\cdot\nabla \psi \dx.
\end{equation*}}
\noindent Next, using density arguments, we obtain
\begin{equation}\label{4.25}
\underset{n\rightarrow+\infty}{\lim}\underset{\mathbb{R}^N}{\int}|\nabla u_n|^{p-2}\nabla u_n\cdot \nabla v \dx = \underset{\mathbb{R}^N}{\int}|\nabla u_\epsilon|^{p-2}\nabla u_\epsilon\cdot\nabla v \dx,\forall v\in \mathbf{X}.
\end{equation}
By similar arguments, as discussed above, we obtain
\begin{align}\label{4.26}
\begin{split}
\underset{n\rightarrow+\infty}{\lim}\underset{\mathbb{R}^N}{\int}|\nabla u_n|^{N-2}\nabla u_n \cdot\nabla v \dx &= \underset{\mathbb{R}^N}{\int}|\nabla u_\epsilon|^{N-2}\nabla u_\epsilon\cdot\nabla v \dx, \forall ~v\in \mathbf{X}, \\ 
\underset{n\rightarrow+\infty}{\lim}\underset{\mathbb{R}^N}{\int}V(x)| u_n|^{N-2} u_n  v \dx &= \underset{\mathbb{R}^N}{\int}V(x)| u_\epsilon|^{N-2} u_\epsilon v \dx,\forall ~v\in \mathbf{X}, \\
\underset{n\rightarrow+\infty}{\lim}\underset{\mathbb{R}^N}{\int}V(x)| u_n|^{p-2} u_n  v \dx &= \underset{\mathbb{R}^N}{\int}V(x)| u_\epsilon|^{p-2} u_\epsilon v \dx,\forall ~v\in \mathbf{X}. \end{split}
\end{align}
From \eqref{4.22}, \eqref{4.25} and \eqref{4.26}, for all $v\in \mathbf{X}$, we have 
\begin{align}\label{4.37}
\underset{\mathbb{R}^N}{\int}|\nabla u_\epsilon|^{p-2}\nabla u_\epsilon.\nabla v \dx &+\underset{\mathbb{R}^N}{\int}|\nabla u_\epsilon|^{N-2}\nabla u_\epsilon\cdot\nabla v \dx \notag \\  &+\underset{\mathbb{R}^N}{\int}V(x)| u_\epsilon|^{p-2} u_\epsilon v \dx+\underset{\mathbb{R}^N}{\int}V(x)| u_\epsilon|^{N-2} u_\epsilon v \dx-\underset{\mathbb{R}^N}{\int}\rho_0 v\dx-\epsilon\underset{\mathbb{R}^N}{\int}gv\dx=0.
\end{align}
The above identity tells us that \ref{solution1} is now verified. Next, we check for condition \ref{solution 2}, i.e., we will show that $\rho_0(x)\in \partial_t F(x,u_{\epsilon}(x)) \text{ a.e. in } \mathbb{R}^N.$
To prove this, it is sufficient to prove that $u_n\rightarrow u_\epsilon \text{ in } \mathbf{X}$ due to Proposition \ref{proposition}. First, we set $\gamma_n=u_n-u_\epsilon$.
Then, the weak convergence of  $u_n$ in $\mathbf{X}$  implies that $\gamma_n\rightharpoonup 0$ in $\mathbf{X}.$ Due to the Br\'ezis-Lieb Lemma-type results \cite{MR699419, MR4857527}, we have
\begin{equation}\label{4.38}
\begin{cases}
\|u_n\|_{W_{V}^{1,p}}^p=\|u_\epsilon\|_{W_{V}^{1,p}}^p+\|\gamma_n\|_{W_{V}^{1,p}}^p+o_n(1),\\
\|u_n\|_{W_{V}^{1,N}}^N=\|u_\epsilon\|_{W_{V}^{1,N}}^N+\|\gamma_n\|_{W_{V}^{1,N}}^N+o_n(1).
\end{cases}
\end{equation}
Hence, we get
\begin{equation*}
o_n(1)=\langle w_n,u_n \rangle =\|u_n\|_{W_{V}^{1,p}}^p+\|u_n\|_{W_{V}^{1,N}}^N-\underset{\mathbb{R}^N}{\int}\rho_n u_n \dx-\epsilon \underset{\mathbb{R}^N}{\int}gu_n\dx.
\end{equation*}
From equation \eqref{4.38}, we obtain
\begin{align}\label{4.40}
o_n(1)&=\|u_\epsilon\|_{W_{V}^{1,p}}^p+\|\gamma_n\|_{W_{V}^{1,p}}^p+o_n(1)+\|u_\epsilon\|_{W_{V}^{1,N}}^N+\|\gamma_n\|_{W_{V}^{1,N}}^N-\underset{\mathbb{R}^N}{\int}\rho_n u_n \dx-\epsilon \underset{\mathbb{R}^N}{\int}gu_n\dx \notag \\ 
&\quad+\underset{\mathbb{R}^N}{\int}\rho_0 u_\epsilon \dx-\underset{\mathbb{R}^N}{\int}\rho_0 u_\epsilon \dx-\epsilon \underset{\mathbb{R}^N}{\int}gu_\epsilon \dx+\epsilon \underset{\mathbb{R}^N}{\int}gu_\epsilon \dx.
\end{align}
From equation \eqref{4.37}, with $v=u_\epsilon$, we have 
\begin{equation}\label{4.41}
0=\|u_\epsilon\|_{W_{V}^{1,p}}^p+\|u_\epsilon\|_{W_{V}^{1,N}}^N-\underset{\mathbb{R}^N}{\int}\rho_0 u_\epsilon \dx-\epsilon \underset{\mathbb{R}^N}{\int}gu_\epsilon \dx.
\end{equation}
Now by using equation \eqref{4.41} and \eqref{4.40}, we obtain
\begin{equation}\label{4.42}
o_n(1)=\|\gamma_n\|_{W_{V}^{1,p}}^p+o_n(1)+\|\gamma_n\|_{W_{V}^{1,N}}^N-\underset{\mathbb{R}^N}{\int}\rho_n (u_n-u_\epsilon) \dx -\epsilon \underset{\mathbb{R}^N}{\int}g(u_n-u_\epsilon)\dx.
\end{equation}
Since  $u_n\rightharpoonup u_\epsilon$ in $\mathbf{X}$,  \eqref{4.42} reduced to 
\begin{equation*}
o_n(1)=\|\gamma_n\|_{W_{V}^{1,p}}^p+\|\gamma_n\|_{W_{V}^{1,N}}^N-\underset{\mathbb{R}^N}{\int}\rho_n (u_n-u_\epsilon) \dx+o_n(1) ~~\text{as }~~n\rightarrow+\infty.
\end{equation*}
Now by assumption \ref{f6}, we get
\begin{equation*}
\left| \underset{\mathbb{R}^N}{\int}\rho_n (u_n-u_\epsilon) \dx\right|\leq \underset{\mathbb{R}^N}{\int}|\rho_n\|\gamma_n|\dx\leq c_1\underset{\mathbb{R}^N}{\int}|u_n|^{N-1}|\gamma_n|\dx+c_2\underset{\mathbb{R}^N}{\int}
|\gamma_n|\Phi(\alpha|u_n|^\frac{N}{N-1})\dx.
\end{equation*}
Upon using H\"older's inequality and embedding $\mathbf{X}\hookrightarrow\hookrightarrow L^{q}(\mathbb{R}^N),$ for all $q\geq 1$, we have 
\begin{equation*}
\underset{\mathbb{R}^N}{\int}|u_n|^{N-1}|\gamma_n|\dx\rightarrow 0 \text{ as } n\rightarrow +\infty.
\end{equation*}
Now choose $r>1$  such that $\frac{1}{r}+\frac{1}{r'}=1$ and $\alpha>\alpha_0$ such that $r\alpha\|u_n\|_{W^{1,N}}^{\frac{N}{N-1}}<\alpha_N$.
Then, H\"older's inequality, the embedding $\mathbf{X}\hookrightarrow\hookrightarrow L^{q}(\mathbb{R}^N),~q\geq 1$ and Lemma \ref{2.1} altogether imply
\begin{equation*}
\underset{\mathbb{R}^N}{\int}
|\gamma_n|\Phi(\alpha|u_n|^\frac{N}{N-1})\dx\rightarrow ~0 \text{ as } n\rightarrow +\infty.
\end{equation*}
This further implies $$\|\gamma_n\|_{W_{V}^{1,p}}^p+\|\gamma_n\|_{W_{V}^{1,N}}^N~~\rightarrow 0 \text{ as } n\rightarrow +\infty.$$
Then, $\gamma_n\rightarrow 0 \text{ in } \mathbf{X}.$ 
Hence,  $\rho_0(x)\in \partial_t F(x,u_{\epsilon}(x)) \text{ a.e. in } \mathbb{R}^N,$ that is \ref{solution 2} is verified.

Last but not least, to conclude the proof, we have to show that condition \ref{solution 3} holds.
One can write $u_\epsilon=u_{\epsilon}^{+}-u_{\epsilon}^{-}$
where
\begin{equation*}
u_{\epsilon}^{+}(x)=
\begin{cases}
 u_\epsilon(x),~~u_\epsilon(x)\geq0,\\
0,~~~~~~~~~u_\epsilon(x)<0. 
\end{cases} \text{ and }
u_{\epsilon}^{-}(x)=
\begin{cases}
 0,~~~~~~~u_\epsilon(x)\geq0,\\
 -u_\epsilon(x),~~~~~u_\epsilon(x)<0. 
\end{cases}
\end{equation*}
Put $v=-u_\epsilon^{-}$ in \eqref{4.37}, we have
\begin{equation*}
0=\|u_\epsilon^{-}\|_{W_V^{1,p}}^p+\|u_\epsilon^{-}\|_{W_V^{1,N}}^N+\epsilon\underset{\mathbb{R}^N}{\int}gu_\epsilon^{-}\dx.
\end{equation*}
This implies $u_\epsilon^{-}\equiv0$, so $u_\epsilon=u_{\epsilon}^{+}\geq 0 \text{ and  } u_\epsilon\not\equiv 0.$ Now we discuss about two possible cases:

\noindent Case 1: Let $t_0=0$, then $m([u_\epsilon>t_0])>0$ always holds due to above argument.\\
Case 2: Let $t_0>0$ and consider $t_0\in (0, t_*)$ where $t_{*}=\frac{N\delta^*}{(N-1)\epsilon(\int_{\mathbb{R}^N}g\dx)}$. Since $\rho_0,u_\epsilon\geq 0$, so due to \eqref{4.37}, we have
\begin{align*}
0=\|u_\epsilon\|_{W_{V}^{1,p}}^p+\|u_\epsilon\|_{W_{V}^{1,N}}^N-\underset{\mathbb{R}^N}{\int}\rho_0 u_\epsilon \dx-\epsilon \underset{\mathbb{R}^N}{\int}gu_\epsilon \dx. \end{align*} It further implies
\begin{align*}
0 \leq \|u_\epsilon\|_{W_{V}^{1,p}}^p+\|u_\epsilon\|_{W_{V}^{1,N}}^N-\epsilon \underset{\mathbb{R}^N}{\int}gu_\epsilon \dx.
\end{align*}
Therefore, we deduce
\begin{equation}\label{4.47}
\epsilon \underset{\mathbb{R}^N}{\int}gu_\epsilon \dx\leq \|u_\epsilon\|_{W_{V}^{1,p}}^p+\|u_\epsilon\|_{W_{V}^{1,N}}^N.
\end{equation}
{On the contrary, we assume that $m([u_\epsilon>t_0])=0$ for some $t_0\in(0,t_{*}).$ Then, $u_{\epsilon} <t_0$ and hence thanks to \ref{f1}, we obtain \(
f(x, u_\epsilon(x))=0~\text{a.e.}~\text{in}~~\mathbb{R}^N.\)}
Hence 
\begin{equation*}
\partial_t F(x,u_\epsilon(x))={0}~\text{a.e.~in}~\mathbb{R}^N.
\end{equation*}
This implies $\rho_0(x)=0~\text{a.e. in}~\mathbb{R}^N.$
By using Lemma \ref{le4.1} and equation \eqref{4.47}, we get
\begin{equation*}
-\delta^*>I_{\epsilon}(u_{\epsilon})=\frac{1}{p}\|u_\epsilon\|_{W_{V}^{1,p}}^p+\frac{1}{N}\|u_\epsilon\|_{W_{V}^{1,N}}^N-\epsilon \underset{\mathbb{R}^N}{\int}gu_\epsilon \dx\geq \frac{1}{N}\left(\|u_\epsilon\|_{W_{V}^{1,p}}^p+\|u_\epsilon\|_{W_{V}^{1,N}}^N\right)-\epsilon \underset{\mathbb{R}^N}{\int}gu_\epsilon \dx.
\end{equation*}
Further, it implies
\begin{equation*}
-\delta^*\geq -\epsilon\left(\frac{N-1}{N}\right)\underset{\mathbb{R}^N}{\int}gu_\epsilon \dx\geq-\epsilon\left(\frac{N-1}{N}\right)t_0\underset{\mathbb{R}^N}{\int}g\dx,
\end{equation*}
and hence, we get
\begin{equation*}
t_0\geq\frac{N\delta^*}{(N-1)\epsilon(\int_{\mathbb{R}^N}g\dx)} =t_{*},
\end{equation*}
which is a contradiction. Therefore, $m([u_\epsilon>t_0])>0$, and the proof is complete. \end{proof}
\section{Existence of the second solution}\label{section5}
In this section, we will establish the solution via a non-smooth version of the Mountain Pass Theorem. We refer to \cite{radulescu1993mountain} for further details.
\begin{lemma}\label{le5.1}
Assume that \ref{f2}, \ref{f3} and \ref{f6} holds. Then, the functional $I_{\epsilon}$ (defined in \eqref{4.1}) satisfies mountain pass geometry.
\end{lemma}
\begin{proof} From Lemma \ref{le4.1}, we conclude that there are $\kappa, \rho, \epsilon_0$ such that for all $\epsilon\in(0,\epsilon_0],$ we have
\begin{equation}\label{5.1}
I_{\epsilon}(u)\geq \kappa,~\forall~\|u\|_\mathbf{X}=\rho.
\end{equation}
Let $\psi_0\in C_{c}^{\infty}(\mathbb{R}^N)\setminus{0}$, $\psi_0>0$ with $\mathrm{supp}(\psi_0)\subset Q$, where $Q$ is fixed compact subset of $\mathbb{R}^N$ in \ref{f3}. For $t>0$,
\begin{equation*}
I_{\epsilon}(t\psi_0)=\frac{t^p}{p}\| \psi_0\|_{W_{V}}^{1,p}+\frac{t^N}{N}\| \psi_0\|_{W_{V}}^{1,N}-\underset{\mathbb{R}^N}{\int} F(x,t\psi_0)\dx-t\epsilon \underset{\mathbb{R}^N}{\int} g\psi_0\dx.
\end{equation*}
From assumption \ref{f3} for $\nu>N$, we deduce
\begin{equation*}
F(x,t)\geq k_1t^{\nu}-k_2,~\text{for}~t\geq0~\text{and} ~\forall~x\in Q.
\end{equation*}
Hence
\begin{equation*}
I_{\epsilon}(t\psi_0)\leq \frac{t^p}{p}\| \psi_0\|_{W_{V}}^{1,p}+\frac{t^N}{N}\| \psi_0\|_{W_{V}}^{1,N}-k_1 t^{\nu}\underset{\mathbb{R}^N}{\int} \psi_0^{\nu}\dx+k_2m(Q)-t\epsilon \underset{\mathbb{R}^N}{\int} g\psi_0\dx.
\end{equation*}
Since $\nu>N$, we obtain 
\begin{equation*}
\underset{t\rightarrow+\infty}{\lim} I_{\epsilon}(t\psi_0)=-\infty.
\end{equation*}
One can choose a sufficiently large $t$ such that 
\begin{equation}\label{5.2}
\phi_0=t\psi_0\in B_{\rho}^c(0).
\end{equation}
From \eqref{5.1} and \eqref{5.2}, the Lemma follows.
Due to equation \eqref{5.1} and  Theorem \ref{thm 2.1}, we get a sequence $\{v_n\}\subset \mathbf{X}$ which satisfies 
\begin{equation}\label{5.3}
I_{\epsilon}(v_n)\rightarrow b_{\epsilon} \text{ and } \lambda_{\epsilon}(v_n)\rightarrow0, \\
\end{equation}
where 
\begin{align*}
b_\epsilon &=\underset {\gamma\in\Gamma~ t\in [0,1]}{\inf~\max}J(\gamma(t))~\text{and}\\
 \Gamma:=\{\gamma &\in C([0,1],\mathbf{X})~|~~ \gamma(0)=0~\text{and}~\gamma(1)=\phi_0 \},
\end{align*} and this completes the proof. \end{proof} 
We will show that $I_{\epsilon}$ verifies the $(PS)_{b_{\epsilon}}$ condition if the parameter $\mu$ given in \ref{f5} is large enough. To do this, we need the following Lemmas.
\begin{lemma}\label{le5.2}
{ Let the assumptions of Lemma \ref{le5.1} hold. Also, suppose assumption \ref{f4} holds.} Then the  sequence $\{v_n\}\subset \mathbf{X}$ obtained in \eqref{5.3} is bounded in $\mathbf{X} $ and there exist $M>0$ such that 
\begin{equation}\label{5.4}
b_\epsilon+\epsilon \|g\|_{*}M\frac{(\beta-1)}{(\beta)}+o_n(1)\geq a\left(\|v_n\|_{W_{V}^{1,p}}^p+\|v_n\|_{W_{V}^{1,N}}^N \right), 
\end{equation}
where $b_{\epsilon}$ is mountain pass level and $a=\left(\frac{1}{N}-\frac{1}{\beta}\right).$
\end{lemma}
\begin{proof} From equation \eqref{5.3}, we have 
\begin{equation*}
I_{\epsilon}(v_n)\rightarrow b_{\epsilon}~\text{and}~\lambda(v_n)\rightarrow 0. 
\end{equation*}
Let $ w_n \in \mathbf{X}^*$ and $\rho_n \in \partial \Upsilon(v_n)$, then we have
\begin{equation}\label{eqnew1}
\begin{aligned}
\begin{cases}
\|w_n\|_{*}&=\lambda_{\epsilon}(v_n)\\
\langle w_n, v_n \rangle &= \|v_n\|_{W_{V}^{1,p}}^p+\|v_n\|_{W_{V}^{1,N}}^N-\underset{\mathbb{R}^N}{\int}\rho_n v_n \dx-\epsilon \underset{\mathbb{R}^N}{\int}gv_n\dx. \end{cases}
\end{aligned}
\end{equation}
Note that  $I_{\epsilon}(v_n)= b_\epsilon +o_n(1)$ and $o_n(1)\|v_n\|_{\mathbf{X}}=\langle w_n, v_n \rangle$. From assumption \ref{f4}, there exists $\beta>N$ verifying 
\begin{equation}\label{eqnew2}
\begin{aligned}\begin{cases}
0\leq\beta F(x,t)\leq\underline{f}(x,t)t,~\forall~t\geq t_0~\text{and}~x\in \mathbb{R}^N, \\
\rho_n(x)\in [\underline{f}(x,v_n(x)), \overline{f}(x,v_n(x))]~\text{a.e. in}~\mathbb{R}^N. \end{cases}
\end{aligned}
\end{equation}
Further, using \eqref{eqnew1} and \eqref{eqnew2}, we deduce
\begin{align*}
b_{\epsilon}+o_n(1)+o_n(1)\|v_n\|_{\mathbf{X}}&\geq I_{\epsilon}(v_n)-\frac{1}{\beta}\langle w_n, v_n \rangle\\ 
&= \frac{1}{p}\|v_n\|_{W_{V}^{1,p}}^p+\frac{1}{N}\|v_n\|_{W_{V}^{1,N}}^N-\underset{\mathbb{R}^N}{\int} F(x,v_n)\dx-\epsilon\underset{\mathbb{R}^N}{\int}gv_n\dx\\ &- \frac{1}{\beta}\|v_n\|_{W_{V}^{1,p}}^p-\frac{1}{\beta}\|v_n\|_{W_{V}^{1,N}}^N+\frac{1}{\beta}\underset{\mathbb{R}^N}{\int}\rho_n v_n \dx+\frac{\epsilon}{\beta}\underset{\mathbb{R}^N}{\int}gv_n\dx\\
&\geq \left(\frac{1}{p}-\frac{1}{\beta}\right)\|v_n\|_{W_{V}^{1,p}}^p+\left(\frac{1}{N}-\frac{1}{\beta}\right)\|v_n\|_{W_{V}^{1,N}}^N-\epsilon\frac{\beta-1}{\beta}\|g\|_{*}\|v_n\|_{\mathbf{X}}.
\end{align*}
So, we obtain
\begin{equation}\label{5.5}
b_{\epsilon}+o_n(1)+o_n(1)\|v_n\|_{\mathbf{X}}\geq a \left(\|v_n\|_{W_{V}^{1,p}}^p+\|v_n\|_{W_{V}^{1,N}}^N\right)-\epsilon\frac{\beta-1}{\beta}\|g\|_{*}\|v_n\|_{\mathbf{X}}.
\end{equation}
Then, \eqref{5.5} implies $\{v_n\}$ is bounded in $\mathbf{X}.$ Suppose on the contrary, $\{v_n\}$ is unbounded in $\mathbf{X} $, then three cases arise:
\begin{align*}
\text{Case 1. } & \underset{n\rightarrow+\infty}{\lim} \|v_n\|_{W_V^{1,p}}=+\infty \text{ and }~\underset{n\rightarrow+\infty}{\lim} \|v_n\|_{W_V^{1,    N}}=+\infty.\\
\text{Case 2. } & \underset{n\rightarrow+\infty}{\lim} \|v_n\|_{W_V^{1,p}}=+\infty \text{ and }~ \underset{n\geq 1}{\sup} \|v_n\|_{W_V^{1,    N}}<+\infty.\\
\text{Case 3. } &\underset{n\geq 1}{\sup} \|v_n\|_{W_V^{1,p}}<+\infty \text{ and } ~\underset{n\rightarrow+\infty}{\lim} \|v_n\|_{W_V^{1,    N}}=+\infty.
\end{align*}  
Suppose, Case 1 holds, then for some constant $K>0$ from \eqref{5.5}, we get
\begin{equation*}
0<a\leq b_{\epsilon}\frac{1}{\|v_n\|_{W_V^{1,p}}^p+\|v_n\|_{W_V^{1,N}}^N}+K\frac{\|v_n\|_{W_V^{1,p}}+\|v_n\|_{W_V^{1,N}}}{\|v_n\|_{W_V^{1,p}}^p+\|v_n\|_{W_V^{1,N}}^N} ~~\text{as}~~ n\rightarrow+\infty.\end{equation*} This implies \begin{equation*}
0<a\leq o_n(1)+K \left(\|v_n\|_{W_V^{1,p}}^{1-p}+\|v_n\|_{W_V^{1,N}}^{1-N}\right)~~\text{as} \quad n\rightarrow+\infty.
\end{equation*}
Since $(c+d)\leq(c^{1-p}+d^{1-N})(c^p+d^N) \text{ for } c,d\geq 0$. As $n\rightarrow+\infty$ in equation \eqref{5.5} we have $0<a\leq0$.
Hence, Case 1 doesn't hold. Suppose Case 2 holds. Then, from \eqref{5.5}, we get
\begin{equation*}
0<a\leq o_n(1)+K\frac{\|v_n\|_{W_V^{1,p}}+\|v_n\|_{W_V^{1,N}}}{\|v_n\|_{W_V^{1,p}}^p}~~\text{as} \quad n\rightarrow+\infty. \end{equation*}
Therefore, $0<a\leq0,$ as $n \rightarrow +\infty$.
Hence, Case 2 also doesn't hold. Using similar arguments as for Case 2, one can conclude that Case 3 also does not hold. Hence, there is $M>0$ such that $\|v_n\|_{\mathbf{X}}\leq M$ and  
\begin{equation*}
b_\epsilon+\epsilon \|g\|_{*}M\frac{(\beta-1)}{(\beta)}+o_n(1)\geq a\left(\|v_n\|_{W_{V}^{1,p}}^p+\|v_n\|_{W_{V}^{1,N}}^N \right)~~\text{as} \quad n\rightarrow+\infty.
\end{equation*}
Hence, the proof is complete. \end{proof} 
\begin{lemma}\label{le5.3}
{Assume that conditions of Lemma \ref{le5.2} are met.} Also, assume that $\epsilon \in (0, \epsilon_0]$ \text{and} $b_{\epsilon}<b_0-\epsilon \left(\frac{\beta-1}{\beta}\right)\|g\|_{*}M,$
where 
\begin{equation}\label{assumple5.3}
\epsilon_0=\frac{\rho^{N-1}}{2^{N+1}N\|g\|_*}~\text{and}~b_0\leq a\left(\min\{1,V_0\}\right)\left(\frac{\alpha_N}{\alpha_0}\right)^{N-1}.
\end{equation}
Then for the $(PS)_{b_{\epsilon}}$ sequence $\{v_n\}$ of $I_\epsilon$ at level $b_\epsilon$ obtained in Lemma \ref{le5.1}, we have
\begin{equation*}
\underset{n\rightarrow +\infty}{\limsup}\|v_n\|_{W^{1,N}}^\frac{N}{N-1}<\frac{\alpha_N}{\alpha_0}.
\end{equation*}
\end{lemma}
\begin{proof}  By  $W_{V}^{1,N}(\mathbb{R}^N)\hookrightarrow W^{1,N}(\mathbb{R}^N)$, we have 
\begin{equation*}
a\left(\min\{1,V_0\}\right)\|v_{n}\|_{W^{1,N}}^N \leq a \|v_n\|_{W_{V}^{1,N}}\leq a\left(\|v_{n}\|_{W^{1,N}}^N+\|v_{n}\|_{W^{1,p}}^p\right).
\end{equation*}
From Lemma \ref{le5.2}, it follows that
\begin{equation*}
a\left(\min\{1,V_0\}\right)\|v_{n}\|_{W^{1,N}}^N \leq\left(\|v_{n}\|_{W^{1,N}}^N+\|v_{n}\|_{W^{1,p}}^p\right)\leq b_\epsilon+\epsilon \|g\|_{*}M\frac{(\beta-1)}{(\beta)}+o_n(1).
\end{equation*}
From the  given assumption \eqref{assumple5.3}, it follows that
\begin{equation*}
 a\left(\min\{1,V_0\}\right)\|v_{n}\|_{W^{1,N}}^N< b_0+o_n(1),
\end{equation*}
and
\begin{equation*}
 a\left(\min\{1,V_0\}\right)\|v_{n}\|_{W^{1,N}}^N< a\left(\min\{1,V_0\}\right)\left(\frac{\alpha_N}{\alpha_0}\right)^{N-1}+o_n(1).
\end{equation*}
This implies,
\begin{equation*}
 \|v_{n}\|_{W^{1,N}}^N< \left(\frac{\alpha_N}{\alpha_0}\right)^{N-1}+o_n(1).
\end{equation*}
So,
\begin{equation*}
 \underset{n\rightarrow +\infty}{\text{lim sup}}\|v_n\|_{W^{1,N}}^\frac{N}{N-1}<\frac{\alpha_N}{\alpha_0}.
\end{equation*}
 Hence, the proof is complete. \end{proof} 
Let $\phi_0=t\psi_0$ be fixed in Lemma \ref{5.1}. Since we have $\gamma>N$ and $\beta>N$, then, there exists $s>0$ such that $s>\frac{N \gamma}{p}.$ Now define
\begin{equation*}
\mu^*=\max~\bigg\{ 1, D_1^\frac{s-N}{N}, D_2^\frac{(\gamma-p)(s-N)}{(ps-N\gamma)}\bigg\},
\end{equation*}
where
\begin{equation*}
D_1=\frac{2}{a\left(\min\{1,V_0\}\right)}\left(\frac{\alpha_N}{2^\frac{N}{N-1}\alpha_0}\right)^{1-N},
\end{equation*} and
\begin{equation*}
D_2=\left(\frac{1}{p}-\frac{1}{\gamma}\right)\left(\frac{2}{\gamma}\right)^\frac{p}{\gamma-p}\left(\frac{\|\phi_0\|_{W_{V}^{1,p}}}{\|\phi_0\|_{L^{\gamma}}}\right)^\frac{p\gamma}{\gamma-p}+\left(\frac{1}{N}-\frac{1}{\gamma}\right)\left(\frac{2}{\gamma}\right)^\frac{N}{\gamma-N}\left(\frac{\|\phi_0\|_{W_{V}^{1,N}} }{\|\phi_0\|_{L^{\gamma}}}\right)^\frac{N\gamma}{\gamma-N}.
\end{equation*}
For any $\mu>\mu^*$,
\begin{equation*}
\begin{split}
D_2<\mu^\frac{ps-N\gamma}{(\gamma-p)(s-N)}=\mu^\frac{ps-\gamma N+pN-pN}{(\gamma-N)(s-N)}=\mu^{\frac{p}{\gamma-p}-\frac{N}{s-N}},
\end{split}
\end{equation*}
which further implies, $\mu ^\frac{N}{s-N}D_2<\mu^\frac{p}{\gamma-p}$, i.e., $\mu ^\frac{N}{s-N}<\frac{\mu^\frac{p}{\gamma-p}}{D_2}.$ { Now we fix $b_0 (\equiv b_0(\mu))$, such that
\begin{equation}\label{5.11}
\begin{split}
\mu ^\frac{N}{s-N}<\frac{2}{b_0}<\frac{\mu^\frac{p}{\gamma-p}}{D_2}.
\end{split}
\end{equation}}From equation \eqref{5.11}, we obtain 
\begin{equation*}
\frac{D_2}{\mu^\frac{p}{\gamma-p}}<\frac{b_0}{2}.
\end{equation*}
For $\mu>\mu^*$, we have $D_1<\mu^\frac{N}{s-N}$, and 
\begin{equation*}
\begin{split}
~~\frac{2}{a\left(\min\{1,V_0\}\right)}\left(\frac{\alpha_N}{2^\frac{N}{N-1}\alpha_0}\right)^{1-N}<\mu^\frac{N}{s-N}<\frac{2}{b_0}.
\end{split}
\end{equation*}
Finally, we get 
\begin{equation}\label{5.12}
b_0<a\left(\text{min}\{1,V_0\}\right)\left(\frac{\alpha_N}{\alpha_0}\right)^{N-1}.
\end{equation}
\begin{lemma}\label{le5.4}
 Suppose \ref{f5} holds and ${\mu>\mu^*}$. Then, there is $\bar{\epsilon}\in(0,\epsilon_0)$ and $t_0\in[0,t_1)$ \text{such that } 
\begin{equation*}
    b_\epsilon<b_0-\epsilon \|g\|_{*}M\frac{\beta-1}{\beta},~\forall~\epsilon\in(0,\bar{\epsilon}]~\text{and } t_0\in[0,t_1),
\end{equation*}
with $b_0(\equiv b_0(\mu))$ satisfying \eqref{5.11}.
\end{lemma}
\begin{proof} Since $\mu>\mu^*$ and $\phi_0=t\psi_0$, So fix $\bar{\epsilon}\in(0,\epsilon_0]$ so small such that 
\begin{equation*}
\frac{b_0}{2}\geq \epsilon \|g\|_{*}M\frac{\beta-1}{\beta}, ~\forall~\epsilon\in(0,\bar{\epsilon}).
\end{equation*}
As $I_{\epsilon}(0)=0$, hence $I_{\epsilon}(t\psi_0)\rightarrow 0~\text{as}~t\rightarrow 0_+.$ So, we have $t_0$ such that 
\begin{equation*}
\underset{t\in[0,t_0]}{\sup}I_{\epsilon}(t\psi_0)<\frac{b_0}{2}\leq b_0-\epsilon \|g\|_{*}M\frac{(\beta-1)}{(\beta)},~\forall~\epsilon\in(0,\bar{\epsilon})~\text{and}~t_0\in[0,\tilde{t}_1),\end{equation*} for some $\tilde{t}_1>0.$
Also we have 
\begin{equation}\label{funcI}
I_{\epsilon}(t\psi_0)=\frac{t^p}{p}\| \psi_0\|_{W_{V}^{1,p}}+\frac{t^N}{N}\| \psi_0\|_{W_{V}^{1,N}}-\underset{\mathbb{R}^N}{\int} F(x,t\psi_0)\dx-t\epsilon \underset{\mathbb{R}^N}{\int} g\psi_0\dx.
\end{equation}
We recall that by \ref{f5}, there exist $\gamma>N$ and $\mu>0$, satisfying
{$$ F(x,t)\geq \mu(t^{\gamma}-t_0^{\gamma})~\text{for}~t\geq t_0,~\forall ~x\in \mathbb{R}^N.$$}So using \ref{f5} in \eqref{funcI}, we further deduce
\begin{align*}
\underset{t\geq t_0}{\sup}I_{\epsilon}(t\psi_0)&\leq \underset{t\geq t_0}{\sup}\left[ \frac{t^p}{p}\| \psi_0\|_{W_{V}^{1,p}}^p+\frac{t^N}{N}\| \psi_0\|_{W_{V}^{1,N}}^N-\underset{\mathbb{R}^N}{\int} F(x,t\psi_0)\dx\right]\\
&\leq \underset{t\geq t_0}{\sup}\left[ \frac{t^p}{p}\| \psi_0\|_{W_{V}^{1,p}}^p+\frac{t^N}{N}\| \psi_0\|_{W_{V}^{1,N}}^N-\mu t^{\gamma}\underset{\mathbb{R}^N}{\int}\psi_0^\gamma \dx+\mu t_0^\gamma m(\mathrm{supp}(\psi_0)) \right]\\
&\leq \underset{t\geq t_0}{\max}\left[\frac{t^p}{p}\| \psi_0\|_{W_{V}^{1,p}}^p-\frac{\mu}{2}t^{\gamma}\|\psi_0\|_{L^{\gamma}}^\gamma\right]+\underset{t\geq t_0}{\max}\left[\frac{t^N}{N}\| \psi_0\|_{W_{V}^{1,N}}^N-\frac{\mu}{2}t^{\gamma}\|\psi_0\|_{L^{\gamma}}^\gamma\right]\\ &\qquad +\mu t_0^\gamma m(\mathrm{supp}(\psi_0)).
\end{align*}
Consider the function $G(t)=\frac{t^p}{p}\| \psi_0\|_{W_{V}^{1,p}}^p-\frac{\mu}{2}t^{\gamma}\|\psi_0\|_{L^{\gamma}}^\gamma$. From elementary calculus, one can easily verify that $G$ attains its maximum at $t=\left(\frac{2\| \psi_0\|_{W_{V}^{1,p}}^p}{\gamma\mu \|\psi_0\|_{L^{\gamma}}^\gamma}\right)^\frac{1}{\gamma-p}$. Similar type of result also hold for $$T(t)=\frac{t^N}{N}\| \psi_0\|_{W_{V}^{1,N}}^N-\frac{\mu}{2}t^{\gamma}\|\psi_0\|_{L^{\gamma}}^\gamma.$$
So, we have 
\begin{align*}
\underset{t\geq t_0}{\sup } I_{\epsilon}(t\psi_0) &\leq \frac{1}{\mu^\frac{p}{\gamma-p}}\left[\left(\frac{1}{p}-\frac{1}{\gamma}\right)\left(\frac{2}{\gamma}\right)^\frac{p}{\gamma-p}\left(\frac{\|\phi_0\|_{W_{V}^{1,p}}}{\|\phi_0\|_{L^{\gamma}}}\right)^\frac{p\gamma}{\gamma-p}\right]\\ &\qquad +\frac{1}{\mu^\frac{N}{\gamma-N}}\left[\left(\frac{1}{N}-\frac{1}{\gamma}\right)\left(\frac{2}{\gamma}\right)^\frac{N}{\gamma-N}\left(\frac{\|\phi_0\|_{W_{V}^{1,N}} }{\|\phi_0\|_{L^{\gamma}}}\right)^\frac{N\gamma}{\gamma-N}\right]
+\mu t_0^{\gamma}m(\mathrm{supp} \psi_0).
\end{align*}
Since $\mu>\mu*\geq 1$ and $\frac{p}{\gamma-p}<\frac{N}{\gamma-N}$. This implies $\mu^\frac{p}{\gamma-p}<\mu^\frac{N}{\gamma-N}.$  So 
\begin{equation}\label{5.14}
\underset{t\geq t_0}{\sup}I_{\epsilon}(t\psi_0)\leq \frac{D_2}{\mu^\frac{p}{\gamma-p}}+\mu t_0^{\gamma}m(\mathrm{supp} \psi_0)<\frac{b_0}{2}+\mu t_0^{\gamma}m(\mathrm{supp} \psi_0).
\end{equation}
Now we choose $t_2>0$ such that 
\begin{equation}\label{5.15}
\mu t_0^{\gamma}m(\mathrm{supp} \psi_0)<\frac{b_0}{2}-\epsilon \|g\|_{*}M\frac{\beta-1}{\beta},~\forall~t_0\in[0,t_2).
\end{equation}
Choose $t_1=\text{min}\{\tilde{t}_1,t_2\}$. From equation \eqref{5.14} and \eqref{5.15} we conclude that
\begin{equation*}
\underset{t\geq t_0}{\sup}I_{\epsilon}(t\psi_0)<b_0-\epsilon \|g\|_{*}M\frac{(\beta-1)}{(\beta)},~\forall~t_0\in[0,t_1)~\text{and}~\epsilon\in(0,\bar\epsilon).
\end{equation*}
Now consider $\gamma(s)=s\phi_0\in \Gamma$, $s\in[0,1]$ such that
\begin{equation*}
b_\epsilon\leq\underset{s\in[0,1]}{\max}I_{\epsilon}(\gamma(s))\leq\underset{s\geq 0}{\max}I_{\epsilon}(\gamma(s))<b_0-\epsilon \|g\|_{*}M\frac{(\beta-1)}{\beta}.
\end{equation*}
Hence, the proof of the Lemma follows.\end{proof}
Now, we are in a situation to prove Theorem \ref{main result 2}.
\begin{proof}[Proof of Theorem \ref{main result 2}] From Lemma \ref{5.1}, we conclude that the functional $I_{\epsilon}$ satisfy mountain pass geometry and obtained a $ (PS)_{b_{\epsilon}}$ sequence $\{v_n\}\subset\mathbf{X}$ which satisfy 
\begin{equation*}
I_{\epsilon}(v_n)\rightarrow b_{\epsilon}~\text{and}~\lambda_{\epsilon}(v_n)\rightarrow 0.
\end{equation*}
Lemma \ref{5.2} implies that $\{v_n\} $ is bounded in $\mathbf{X}.$ This implies  $v_n\rightharpoonup v_\epsilon$, so up to a subsequence still denoted by itself, we have 
$$\begin{cases}
 v_n\rightarrow v_\epsilon ~~\text{a.e. ~~in}~~ \mathbb{R}^N,\\
 v_n\rightharpoonup v_\epsilon ~~\text{in}~~W_{V}^{1,t}(\mathbb{R}^N),~t\in\{p,N\},\\
 v_n\rightarrow v_{\epsilon}~~\text{in}~~L^q(\mathbb{R}^N),~\forall~q\in[1,+\infty).
 \end{cases}$$
Therefore, by \eqref{5.12} and Lemmas \ref{le5.3}--\ref{le5.4}, we deduce 
 \begin{equation*}
  \underset{n\rightarrow +\infty}{\mathrm{lim~sup}}\|v_n\|_{W^{1,N}}^\frac{N}{N-1}<\frac{\alpha_N}{\alpha_0}.
 \end{equation*}
Also, from Lemma \ref {l4.3}, it follows that $\nabla v_n\rightarrow \nabla v_{\epsilon}$ a.e. in $\mathbb{R}^N.$  Using the same argument as in Theorem \ref{main result 1}, we can show that $v_n\rightarrow v_\epsilon$ in $\mathbf{X}.$  So from Lemma \ref{5.4}, we conclude that $I_{\epsilon}$ admits a critical point $v_\epsilon$ upto a mountain pass level $b_\epsilon$ estimated in Lemma \ref{5.4}. The non-negativity of the solution $v_\epsilon$, conditions \ref{solution 2}, \ref{solution 3} can be proved using the same arguments as used in the proof of Theorem \ref{main result 1}. Hence, $v_\epsilon\in \mathbf{X}$ is the solution for problem \eqref{main problem} such that $I_\epsilon(v_\epsilon)=b_\epsilon>0.$\end{proof}
Finally, we include the proof of Theorem \ref{main result 3}.
\begin{proof}[Proof of Theorem \ref{main result 3}] The proof of this theorem is a direct consequence of Theorems \ref{main result 1} and \ref{main result 2}. Hence, we obtain two distinct nonnegative solutions in $\mathbf{X}$ such that 
$I_{\epsilon}(u_\epsilon)=c_{\epsilon}<0<b_\epsilon=I_\epsilon(v_\epsilon). $  \end{proof}
\section*{Acknowledgements}{First of all, the authors would like to express their gratitude to the anonymous referees for their meaningful comments and suggestions that contributed to improving the manuscript.} AS expresses his heartfelt gratitude for the DST-INSPIRE Grant DST/INSPIRE/04/2018/002208, funded by the Government of India.
\bibliography{ref}
\bibliographystyle{abbrv} 
\end{document}